\documentclass[a4paper]{amsart}
\usepackage{graphicx}
\usepackage{amssymb}
\usepackage{amsmath}
\usepackage{amsthm}
\usepackage{amscd}
\usepackage[all,2cell]{xy}
\usepackage[pagebackref,colorlinks]{hyperref}

\makeatletter
\newcommand{\colim@}[2]{%
  \vtop{\m@th\ialign{##\cr
    \hfil$#1\operator@font colim$\hfil\cr
    \noalign{\nointerlineskip\kern1.5\ex@}#2\cr
    \noalign{\nointerlineskip\kern-\ex@}\cr}}%
}
\newcommand{\colim}{%
  \mathop{\mathpalette\colim@{}}
  }
\makeatother

\UseAllTwocells \SilentMatrices
\newtheorem{thm}{Theorem}[section]

\newtheorem{cor}[thm]{Corollary}

\newtheorem{lem}[thm]{Lemma}
\newtheorem{exm}[thm]{Example}

\newtheorem{prop}[thm]{Proposition}

\theoremstyle{definition}
\newtheorem{defn}[thm]{Definition}
\theoremstyle{remark}
\newtheorem{rem}[thm]{\bf Remark}
\numberwithin{equation}{section}

\begin{document}
\title[Differential graded enhancements of singularity categories]{Differential graded enhancements of singularity categories}
\author[Chen, Wang] {Xiao-Wu Chen, Zhengfang Wang}

%\thanks{$^*$ The corresponding author}
%\thanks{}
\date{revised version, \today}

\thanks{xwchen$\symbol{64}$mail.ustc.edu.cn, zhengfangw$\symbol{64}$gmail.com}

\maketitle

\dedicatory{}%
\commby{}%

\begin{abstract}
The  singularity category of a ring detects the homological singularity of the given ring, and appears in many different contexts. We describe two different dg enhancements of the singularity category, that is, the Vogel dg category and the singular Yoneda dg category. These two dg enhancements turn out to be quasi-equivalent. We report some progress on the Singular Presilting Conjecture.
\end{abstract}

\section{Introduction}

Let $R$ be a left coherent ring. Denote by $\mathbf{D}^b(\text{$R$-mod})$ the bounded derived category of the abelian category consisting of finitely presented  left $R$-modules.   The singularity category $\mathbf{D}_{\rm sg}(R)$ of $R$ is defined to be the Verdier quotient category \cite{Ver} of $\mathbf{D}^b(\text{$R$-mod})$ by the full triangulated subcategory of perfect complexes. The singularity category detects the homological singularity of $R$ in the following sense: if $R$ has finite global dimension, then $\mathbf{D}_{\rm sg}(R)$ vanishes.

The  singularity category was first introduced in \cite{Buc} and then rediscovered in \cite{Orl} in the geometric setting, which is motivated by the homological mirror symmetry conjecture \cite{Kon}. The singularity category appears naturally in a number of different subjects, such as matrix factorizations \cite{Eis}, integral and modular representations of finite groups \cite{Buc, Ric}, Gorenstein homological algebra \cite{ABr, Buc, EJ1}, noncommutative algebraic geometry \cite{Orl2},  weighted projective lines \cite{Len} and cluster categories \cite{AIR}.

We will abbreviate `differential graded' as dg. By a dg enhancement \cite{BK2} of a triangulated category, we mean a pretriangulated dg category whose zeroth cohomology yields the given triangulated category. It is well known that any triangulated category appearing naturally in algebra has a dg enhancement; compare \cite[Subsection~3.6]{Kel06}. Moreover, the dg enhancement contains more information and has more  invariants such as the Hochschild cohomology ring.

The dg singularity category \cite{Kel18, BRTV, BrDy} of $R$ is defined to be the dg quotient category of the bounded dg derived category by the full dg subcategory of perfect complexes. It is a canonical dg enhancement of $\mathbf{D}_{\rm sg}(R)$. We mention that the dg singularity category plays a crucial role in the study \cite{Kel18, HK, JM} of the Donovan-Wemyss conjecture on isolated compound Du Val singularities.

Since the dg singularity category is defined as a dg quotient category \cite{Kel99, Dri}, it is  quite hard to handle. Therefore,  it is of interest to  have more explicit dg models. In this paper, we describe two different  dg enhancements of the singularity category, that is, the \emph{Vogel dg category} and the \emph{singular Yoneda dg category}.

 The idea of the Vogel dg category is implicit in  an unpublished paper of Vogel, which generalizes the Tate cohomology \cite{Tate, Buc} of Gorenstein rings to arbitrary rings; see also \cite{Goi, BC, Mis}.  The construction of the singular Yoneda dg category \cite{CW} relies on  bar resolutions and noncommutative differential forms.

The paper is organized as follows. In Section~2, we obtain  semi-orthogonal decompositions consisting of Verdier quotient categories of homotopy categories of complexes; see Propositions~\ref{prop:semi-ortho} and ~\ref{prop:semi-ortho2}. We define the Vogel dg category in Section~3, whose zeroth cohomology is closely related to the quotient categories studied in Section~2 via an orthogonal decomposition; see Theorem~\ref{thm:Vogel}.

In Section~4, we first point out in Proposition~\ref{prop:sing-V} that the Vogel dg category is quasi-equivalent to the dg singularity category. We give a full proof of a slight generalization of Buchweitz's theorem on the singularity category of a Gorenstein ring; see Theorem~\ref{thm:Buc}. In Subsection~4.3, we report some progress in \cite{CLZZ} on the Singular Presilting Conjecture \cite{CHQW, IY}, which implies the well-known Auslander-Reiten Conjecture \cite{AR75}.

  In Section~5, we recall the construction of the singular Yoneda dg category \cite{CW}.  Proposition~\ref{prop:sing-SY} states that the singular Yoneda dg category is quasi-equivalent to the dg singularity category. In summary, although the Vogel dg category and the singular Yoneda dg category are quite different, they turn out to be quasi-equivalent.

We emphasize that almost all the results in Sections~4 and 5 exist already in the literature. Some of the results in Sections~2 and 3 seem to be new.

\section{Quotient categories and singularity categories}

In this section, we obtain  semi-orthogonal decompositions of certain Verdier quotient categories of homotopy categories of complexes; see Propositions~\ref{prop:semi-ortho} and ~\ref{prop:semi-ortho2}.

 In what follows, we assume that $\mathfrak{a}$ is an additive category. A (cochain) complex in $\mathfrak{a}$ is usually denoted by $X=(X^n, d_X^n)_{n\in \mathbb{Z}}$, where the differentials $d_X^n\colon X^n\rightarrow X^{n+1}$ satisfy $d_X^{n+1}\circ d_X^n=0$. For each integer $n$, we denote by $\Sigma^n(X)$ the degree $n$ shift of $X$. Denote by $C(\mathfrak{a})$ and $\mathbf{K}(\mathfrak{a})$ the category of complexes in $\mathfrak{a}$ and the homotopy category, respectively.

 For a complex $X$ and $n\in \mathbb{Z}$, we have the following brutal truncations:
 $$\sigma_{\geq n}(X)= \cdots \rightarrow 0\rightarrow X^n \stackrel{d_X^n}\rightarrow X^{n+1}\stackrel{d_X^{n+1}}\rightarrow X^{n+1}\rightarrow \cdots$$
 and
 $$\sigma_{<n}(X)=\cdots \rightarrow X^{n-3}\stackrel{d_X^{n-3}}\rightarrow X^{n-2}\stackrel{d_X^{n-2}}\rightarrow X^{n-1}\rightarrow 0\rightarrow \cdots.$$
 Moreover, we have a canonical exact triangle in $\mathbf{K}(\mathfrak{a})$:
 \begin{align}\label{tri:can}
 \sigma_{\geq n}(X)\xrightarrow{{\rm inc}_n} X\xrightarrow{{\rm pr}_n}  \sigma_{<n}(X)\longrightarrow \Sigma \sigma_{\geq n}(X).
 \end{align}
 Here, ${\rm inc}_n$ and ${\rm pr}_n$ denote the corresponding inclusion and projection, respectively. The unnamed cochain map $\sigma_{<n}(X)\rightarrow \Sigma\sigma_{\geq n}(X)$ is induced by the differential $d_X^{n-1}\colon X^{n-1}\rightarrow X^n$.

 A complex $X$ is bounded-above (\emph{resp}., bounded-below) if $X^n=0$ for $n$ (\emph{resp}., $-n$) sufficiently large. A complex is bounded if it is both bounded-above and bounded-below. We denote by $C^{-}(\mathfrak{a})$, $C^{+}(\mathfrak{a})$  and $C^b(\mathfrak{a})$ the full subcategories of $C(\mathfrak{a})$ formed by bounded-above, bounded-below and bounded complexes, respectively. The corresponding homotopy categories are denoted by $\mathbf{K}^{-}(\mathfrak{a})$, $\mathbf{K}^{+}(\mathfrak{a})$  and $\mathbf{K}^b(\mathfrak{a})$, respectively.

In what follows, we consider the Verdier quotient category $\mathbf{K}(\mathfrak{a})/\mathbf{K}^b(\mathfrak{a})$.

\begin{lem}\label{lem:quot}
Let $X$ and $Y$ be two complexes. The following statements hold.
\begin{enumerate}
\item Assume that $Y$ is bounded-above. Then we have a natural isomorphism
$${\rm Hom}_{\mathbf{K}(\mathfrak{a})/\mathbf{K}^b(\mathfrak{a})}(X, Y)\simeq \colim\limits_{n\to +\infty} {\rm Hom}_{\mathbf{K}(\mathfrak{a})}(X, \sigma_{\leq -n}(Y)),$$
where the structure maps in the colimit are induced by the projections $\sigma_{\leq -n}(Y)\rightarrow \sigma_{\leq -(n+1)}(Y)$.
\item Assume that $X$ is bounded-below. Then we have a natural isomorphism
$${\rm Hom}_{\mathbf{K}(\mathfrak{a})/\mathbf{K}^b(\mathfrak{a})}(X, Y)\simeq \colim\limits_{n\rightarrow +\infty} {\rm Hom}_{\mathbf{K}(\mathfrak{a})}(\sigma_{\geq n}(X), Y),$$
where the structure maps are induced by the inclusions $\sigma_{\geq n+1}(X)\rightarrow \sigma_{\geq n}(X)$.
\end{enumerate}
\end{lem}

\begin{proof}
For (1), we recall from \cite[(2.2.1.1)]{Ver} an isomorphism
$${\rm Hom}_{\mathbf{K}(\mathfrak{a})/\mathbf{K}^b(\mathfrak{a})}(X, Y)\simeq {\rm colim}_\mathcal{I}\; {\rm Hom}_{\mathbf{K}(\mathfrak{a})}(X, Y'),$$
where the colimit is taken over a filtered category $\mathcal{I}$. Here, $\mathcal{I}$ is the full subcategory of the coslice category $(Y\downarrow \mathbf{K}(\mathfrak{a}))$ formed by morphisms $s\colon Y\rightarrow Y'$ such that its mapping cone ${\rm Cone}(s)$ is isomorphic to a bounded complex in $\mathbf{K}(\mathfrak{a})$.

Since $Y$ is bounded-above, the projections ${\rm pr}_{-n}\colon Y\rightarrow \sigma_{<-n}(Y)$ are objects in $\mathcal{I}$. Denote by $\mathcal{J}$ the non-full subcategory of $\mathcal{I}$ formed by the projections ${\rm pr}_{-n}$ and the natural morphisms ${\rm pr}_{-n}\rightarrow {\rm pr}_{-(n+1)}$. Since any morphism from a bounded complex to $Y$ factors through the inclusion $\sigma_{\geq -n}(Y)\rightarrow Y$ for $n$ sufficiently large, it follows that any object $s\colon Y\rightarrow Y'$ in $\mathcal{I}$ admits a morphism to ${\rm pr}_{-n}$ for $n$ sufficiently large. Moreover, the morphism $Y'\rightarrow \sigma_{<-n}(Y)$, which yields the morphism $s\rightarrow {\rm pr}_{-n}$ in  $\mathcal{I}$, is unique in $\mathbf{K}(\mathfrak{a})$, since ${\rm Hom}_{\mathbf{K}(\mathfrak{a})}({\rm Cone}(s), \sigma_{<-n}(Y))=0$ for $n$ sufficiently large. This uniqueness implies that $\mathcal{J}$  is cofinal in $\mathcal{I}$; see \cite[Definition~2.5.1]{KS}. Then we infer (1) by \cite[Proposition~2.5.2(v)]{KS}. Dually, we prove (2).
\end{proof}

\begin{cor}\label{cor:ortho}
Let $X$ be a bounded-below complex and $Y$ a bounded-above complex. Then we have
$${\rm Hom}_{\mathbf{K}(\mathfrak{a})/\mathbf{K}^b(\mathfrak{a})}(X, Y)=0.$$
\end{cor}

\begin{proof}
We have ${\rm Hom}_{\mathbf{K}(\mathfrak{a})}(X, \sigma_{\leq -n}(Y))=0$ if $n$ is sufficiently large. Then the result follows immediately from  Lemma~\ref{lem:quot}(1).
\end{proof}

Let $\mathcal{T}$ be a triangulated category. For a triangulated subcategory $\mathcal{X}$, the orthogonal subcategories are defined to be
$$\mathcal{X}^\perp=\{Y\in \mathcal{T}\;|\; {\rm Hom}_\mathcal{T}(X, Y)=0  \mbox{ for all } X\in \mathcal{X}\}$$
 and
 $$^\perp\mathcal{X}=\{Y\in \mathcal{T}\;|\; {\rm Hom}_\mathcal{T}(Y, X)=0  \mbox{ for all } X\in \mathcal{X}\}.$$
 Both $\mathcal{X}^\perp$ and $^\perp\mathcal{X}$ are triangulated subcategories of $\mathcal{T}$.

Recall from \cite{Bon, BK} that a \emph{semi-orthogonal decomposition} $\mathcal{T}=\langle \mathcal{X}, \mathcal{Y}\rangle$ consists of two full triangulated subcategories $\mathcal{X}$ and $\mathcal{Y}$ subject to the following conditions:
\begin{enumerate}
\item ${\rm Hom}_\mathcal{T}(Y, X)=0$ for any $Y\in\mathcal{Y}$ and $X\in \mathcal{X}$;
\item any object $T$ in $\mathcal{T}$ fits into an exact triangle
$$Y\longrightarrow T\longrightarrow X\longrightarrow \Sigma(Y)$$
with some $Y\in \mathcal{Y}$ and $X\in \mathcal{X}$.
\end{enumerate}
The semi-orthogonal decomposition is \emph{orthogonal} if ${\rm Hom}_\mathcal{T}(X, Y)=0$ for any  $X\in\mathcal{X}$ and $Y\in \mathcal{Y}$, in which case, we write $\mathcal{T}=\mathcal{X}\times\mathcal{Y}$. We mention that semi-orthogonal decompositions are essentially equivalent to stable $t$-structures in the sense of \cite[Definition~9.14]{Miy} and Bousfield localizations of triangulated categories in \cite[Chapter~9]{Nee}.

\begin{rem}\label{rem:ortho}
Let  $\mathcal{T}=\langle \mathcal{X}, \mathcal{Y}\rangle$ be a semi-orthogonal decomposition. Then it is orthogonal if and only if each object $T$ in $\mathcal{T}$ admits a decomposition $T\simeq Y\oplus X$ with $Y\in \mathcal{Y}$ and $X\in \mathcal{X}$.
\end{rem}

The following facts are well known; see \cite[Lemma~3.1]{Bon} and \cite[Proposition~1.6]{BK}.

\begin{lem}\label{lem:semi-ortho}
Assume that $\mathcal{T}=\langle \mathcal{X}, \mathcal{Y}\rangle$ is a semi-orthogonal decomposition. Then the following statements hold.
\begin{enumerate}
\item We have $\mathcal{X}=\mathcal{Y}^\perp$ and $\mathcal{Y}={^\perp\mathcal{X}}$.
\item The canonical functors $\mathcal{X}\rightarrow \mathcal{T}/\mathcal{Y}$ and $\mathcal{Y}\rightarrow \mathcal{T}/\mathcal{X}$ are triangle equivalences.
\end{enumerate}
\end{lem}

We view $\mathbf{K}^{-}(\mathfrak{a})/\mathbf{K}^b(\mathfrak{a})$ and $\mathbf{K}^{+}(\mathfrak{a})/\mathbf{K}^b(\mathfrak{a})$ as full triangulated subcategories of $\mathbf{K}(\mathfrak{a})/\mathbf{K}^b(\mathfrak{a})$.

We mention that one might deduce the following semi-orthogonal decomposition by using a general result \cite[Theorem~B]{JK}.

\begin{prop}\label{prop:semi-ortho}
Keep the notation as above. Then we have a semi-orthogonal decomposition
$$\mathbf{K}(\mathfrak{a})/\mathbf{K}^b(\mathfrak{a})=\langle \mathbf{K}^{-}(\mathfrak{a})/\mathbf{K}^b(\mathfrak{a}),  \mathbf{K}^{+}(\mathfrak{a})/\mathbf{K}^b(\mathfrak{a})\rangle.$$
Consequently, the following canonical functors
$$\mathbf{K}^{-}(\mathfrak{a})/\mathbf{K}^b(\mathfrak{a})\rightarrow \mathbf{K}(\mathfrak{a})/\mathbf{K}^{+}(\mathfrak{a}) \quad \text{and}\quad \mathbf{K}^{+}(\mathfrak{a})/\mathbf{K}^b(\mathfrak{a})\rightarrow \mathbf{K}(\mathfrak{a})/\mathbf{K}^{-}(\mathfrak{a}) $$
are both triangle equivalences.
\end{prop}

\begin{proof}
The semi-orthogonal decomposition follows from  Corollary~\ref{cor:ortho} and the exact triangle (\ref{tri:can}). By \cite[Chapitre~2, Proposition~2.3.1.c)]{Ver}, we have  the following well-known equivalence
$$\mathbf{K}(\mathfrak{a})/\mathbf{K}^{+}(\mathfrak{a})\stackrel{\sim}\longrightarrow {(\mathbf{K}(\mathfrak{a})/\mathbf{K}^b(\mathfrak{a}))}/{(\mathbf{K}^{+}(\mathfrak{a})/
\mathbf{K}^b(\mathfrak{a}))}.$$
Then we deduce the first equivalence in the consequence by applying Lemma~\ref{lem:semi-ortho}(2) to the semi-orthogonal decomposition above. Similarly, we have the second one.
\end{proof}

In what follows, we assume that $R$ is an arbitrary ring.  Denote by  $R\mbox{-Mod}$ the category of left $R$-modules and by $R\mbox{-mod}$ its full subcategory formed by finitely presented modules. Denote by $R\mbox{-Proj}$ the category of projective $R$-modules and by $R\mbox{-proj}$ the full subcategory formed by finitely generated ones. Denote by $\mathbf{K}^{-, b}(R\mbox{-proj})$ the full subcategory of $\mathbf{K}^{-}(R\mbox{-proj})$ consisting of bounded-above complexes with bounded cohomology. Similarly, we have the category $\mathbf{K}^{-,b}(R\mbox{-Proj})$.

The following well-known notion is due to \cite{Buc}; see also \cite{Orl}.

\begin{defn}
The \emph{singularity category} of $R$ is the following Verdier quotient category
$$\mathbf{D}_{\rm sg}(R)=\mathbf{K}^{-, b}(R\mbox{-proj})/{\mathbf{K}^{b}(R\mbox{-proj})}.$$
\end{defn}

\begin{rem}\label{rem:bigsingularitycategory}
(1) The \emph{big singularity category} of $R$ is defined to be
$$\mathbf{D}'_{\rm sg}(R)=\mathbf{K}^{-, b}(R\mbox{-Proj})/{\mathbf{K}^{b}(R\mbox{-Proj})}.$$
By \cite[Remark~3.6]{Chen11}, the canonical functor $\mathbf{D}_{\rm sg}(R)\rightarrow \mathbf{D}_{\rm sg}'(R)$ is fully faithful; compare \cite[Proposition~1.13]{Orl}.

(2) Assume that $R$ is left coherent. Then $R\mbox{-mod}$ is abelian. We might identify the bounded derived category $\mathbf{D}^b(R\mbox{-mod})$ with $\mathbf{K}^{-, b}(R\mbox{-proj})$. Therefore, the singularity category $\mathbf{D}_{\rm sg}(R)$ might be defined as $\mathbf{D}^b(R\mbox{-mod})/{\mathbf{K}^{b}(R\mbox{-proj})}$. Similarly, for an arbitrary ring $R$, we might define $\mathbf{D}'_{\rm sg}(R)$ as $\mathbf{D}^b(R\mbox{-Mod})/{\mathbf{K}^{b}(R\mbox{-Proj})}$.
\end{rem}

Let $P$ be a complex in  $\mathbf{K}^{-, b}(R\mbox{-proj})$. Then there exists a sufficiently large $n_0$ such that for any $n\geq n_0$, the cocycle $Z^{-n}(P)$ is finitely presented and $\sigma_{< -n}(P)$ is a shifted projective resolution of $Z^{-n}(P)$. For any such $n$, the projection
\begin{align}\label{iso:proj}
P\longrightarrow \sigma_{<-n}(P)
\end{align}
is an isomorphism in $\mathbf{D}_{\rm sg}(R)$. In other words, any object in the singularity category is given by a shifted projective resolution of a finitely presented module. Denote by $R\mbox{-\underline{mod}}$ the stable category of $R\mbox{-mod}$ modulo projective modules; see \cite[Chapter~IV.1]{ARS}.

\begin{lem}\label{lem:iso-in-sg}
Let $P$ and $Q$ be two complexes in  $\mathbf{K}^{-, b}(R\mbox{-{\rm proj}})$. Then $P$ and $Q$ are isomorphic in $\mathbf{D}_{\rm sg}(R)$ if and only if $Z^{-n}(P)$ and $Z^{-n}(Q)$ are isomorphic in $R\mbox{-\underline{\rm mod}}$ for any sufficiently large $n$.
\end{lem}

\begin{proof}
The following observation implies the  ``only if" part: if there is a cochain map $f\colon P\rightarrow P'$ whose mapping cone is homotopical to a bounded complex, then $Z^{-n}(P)$ and $Z^{-n}(P')$ are isomorphic in $R\mbox{-\underline{\rm mod}}$ for any sufficiently large $n$. The ``if" part follows from the isomorphism (\ref{iso:proj}) applied to $P$ and $Q$, respectively.
\end{proof}

For a complex $P$  in $\mathbf{K}(R\mbox{-proj})$, we consider the \emph{dual complex} $P^*={\rm Hom}_R(P, R)$, which is a complex of finitely generated projective right $R$-modules. We view $P^*$ as an object in $\mathbf{K}(R^{\rm op}\mbox{-proj})$. Here, $R^{\rm op}$ denotes the opposite ring of $R$, and we identify right $R$-modules with left $R^{\rm op}$-modules.

The following treatment might be compared with \cite[Subsection~6.1]{JK}. We consider the following full triangulated subcategory of $\mathbf{K}(R\mbox{-proj})$:
$$\mathbf{K}^B(R\mbox{-proj})=\{P\in \mathbf{K}(R\mbox{-proj})\; |\; H^{-n}(P)=0=H^{-n}(P^*) \mbox{ for } n \gg 0\}.$$
We have $\mathbf{K}^b(R\mbox{-proj})\subseteq \mathbf{K}^B(R\mbox{-proj})$, and consider the Verdier quotient category $\mathbf{K}^B(R\mbox{-proj})/{\mathbf{K}^b(R\mbox{-proj})}$.

We will  adopt semi-orthogonal decompositions in a slightly more general setting. Let $\mathcal{T}=\langle \mathcal{X}, \mathcal{Y}\rangle$ be a semi-orthogonal decomposition. If $\mathcal{X}'$ (\emph{resp}. $\mathcal{Y}'$) is a triangulated category which is triangle equivalent to $\mathcal{X}$ (\emph{resp}. $\mathcal{Y}$), we still write $\mathcal{T}=\langle \mathcal{X}', \mathcal{Y}'\rangle$. For a triangulated category $\mathcal{X}$, we denote by $\mathcal{X}^{\rm op}$ its opposite category, which is also naturally triangulated.

\begin{prop}\label{prop:semi-ortho2}
We have a semi-orthogonal decomposition  $$\mathbf{K}^B(R\mbox{-{\rm proj}})/{\mathbf{K}^b(R\mbox{-{\rm proj}})}=\langle \mathbf{D}_{\rm sg}(R), \mathbf{D}_{\rm sg}(R^{\rm op})^{\rm op}\rangle.$$
\end{prop}

\begin{proof}
The semi-orthogonal decomposition in Proposition~\ref{prop:semi-ortho} restricts to the following one
$$\frac{\mathbf{K}^B(R\mbox{-proj})}{{\mathbf{K}^b(R\mbox
{-proj})}}=\langle \frac{\mathbf{K}^B(R\mbox{-proj})\cap \mathbf{K}^-(R\mbox{-proj})}{{\mathbf{K}^b(R\mbox{-proj})}}, \frac{\mathbf{K}^B(R\mbox{-proj})\cap \mathbf{K}^+(R\mbox{-proj})}{{\mathbf{K}^b(R\mbox{-proj})}}  \rangle.$$
For the first factor in the decomposition, we just observe that
$$\mathbf{K}^B(R\mbox{-proj})\cap \mathbf{K}^-(R\mbox{-proj})=\mathbf{K}^{-, b}(R\mbox{-proj}).$$
Therefore, we identify the first factor with $\mathbf{D}_{\rm sg}(R)$. For the second factor, we use the following duality of categories
$${\rm Hom}_{R^{\rm op}}(-, R)\colon \mathbf{K}^{-, b}(R^{\rm op}\mbox{-proj})\longrightarrow \mathbf{K}^B(R\mbox{-proj})\cap \mathbf{K}^+(R\mbox{-proj}),$$
which sends $\mathbf{K}^{b}(R^{\rm op}\mbox{-proj})$ to $\mathbf{K}^{b}(R\mbox{-proj})$. The duality allows us to identify the second factor with $\mathbf{D}_{\rm sg}(R^{\rm op})^{\rm op}$. Then we are done.
\end{proof}

The  following example shows that the semi-orthogonal decomposition in Proposition~\ref{prop:semi-ortho2} is not orthogonal in general. It follows that the same holds for the one in Proposition~\ref{prop:semi-ortho}.

\begin{exm}
{\rm Let $R$ be a quasi-Frobenius ring and $M$ be a finitely generated non-projective $R$-module. Take a \emph{complete resolution} $P$ of $M$, that is $P=(P^n, d_P^n)_{n\in \mathbb{Z}}$ is an acyclic complex of finitely generated projective $R$-modules with its $(-1)$-th cocycle $Z^{-1}(P)\simeq M$. The differential $d_P^{0}\colon P^0\rightarrow P^1$ induces a cochain map $\phi\colon \sigma_{\leq 0}(P)\rightarrow \Sigma\sigma_{\geq 1}(P)$.

We claim that $\phi$ yields a nonzero morphism in $\mathbf{K}(R\mbox{-proj})/\mathbf{K}^b(R\mbox{-proj})$. This implies that the semi-orthogonal decomposition in Proposition~\ref{prop:semi-ortho2} is not orthogonal.

For the claim,  we assume the converse. Then, when viewed as a morphism in $\mathbf{K}(R\mbox{-proj})$, $\phi$ factors through an object $Q\in \mathbf{K}^b(R\mbox{-proj})$.  We observe that $\phi$ is a quasi-isomorphism and becomes an isomorphism in the derived category $\mathbf{D}(R\mbox{-Mod})$. Indeed,  both $\sigma_{\leq 0}(P)$ and $\Sigma\sigma_{\geq 1}(P)$ are quasi-isomorphic to $M$. It follows that the identity map on $M$ factors through $Q$ in $\mathbf{D}(R\mbox{-Mod})$. This implies that in $\mathbf{D}(R\mbox{-Mod})$, $M$ is isomorphic to a bounded complex of finitely generated projective modules. This is impossible, as the projective dimension of $M$ is infinite.
}\end{exm}

\section{The Vogel dg category}

For an additive category, we will introduce its Vogel dg category; see Definition~\ref{defn:V}. We obtain an orthogonal decomposition of the homotopy category of the Vogel dg category; see Theorem~\ref{thm:Vogel}.

\subsection{Preliminaries on dg categories}

In this subsection, we recall from \cite{Kel94, Dri} basic facts on dg categories.

Let $\mathcal{C}$ be a dg category. For two objects $X$ and $Y$, the morphism complex is usually denoted by $\mathcal{C}(X, Y)$, whose differential is denoted by $d_\mathcal{C}$ or simply by $d$. By default, we only consider homogeneous morphisms in any dg category. For any $f\in \mathcal{C}(X, Y)^n$, we write $|f|=n$. A morphism $f\colon X\rightarrow Y$ is called \emph{closed}, if $d_\mathcal{C}(f)=0$. A closed isomorphism of degree zero is called a \emph{dg-isomorphism}.

The \emph{ordinary category} $Z^0(\mathcal{C})$ of $\mathcal{C}$ is a pre-additive category, which has the same objects as $\mathcal{C}$ has,  and whose morphisms are precisely closed morphisms of degree zero in $\mathcal{C}$. The \emph{homotopy category} $H^0(\mathcal{C})$ is a factor category of $Z^0(\mathcal{C})$, whose morphisms are given by the zeroth cohomology groups of the morphism complexes in $\mathcal{C}$. In other words, we have
$${\rm Hom}_{Z^0(\mathcal{C})}(X, Y)=Z^0(\mathcal{C}(X, Y)) \quad \text{and} \quad {\rm Hom}_{H^0(\mathcal{C})}(X, Y)=H^0(\mathcal{C}(X, Y)).$$
A closed morphism of degree zero is called \emph{homotopy equivalence} if it becomes an isomorphism in $H^0(\mathcal{C})$. Clearly, a dg-isomorphism is a homotopy equivalence.

An object $X$ in $\mathcal{C}$ is called \emph{contractible} if there exists an endomorphism $\epsilon$ of $X$ of degree $-1$ with $d_\mathcal{C}(\epsilon)={\rm Id}_X$. This is equivalent to the condition that $X$ is the zero object in $H^0(\mathcal{C})$.

The prototypical example of dg categories is as follows.

\begin{exm}\label{exm:Hom-complex}
 Let $\mathfrak{a}$ be an additive category.  For two complexes $X$ and $Y$, the Hom-complex ${\rm Hom}_\mathfrak{a}(X, Y)$ is a complex of abelian groups whose $n$-th component is given by the following infinite product
$${\rm Hom}_\mathfrak{a}(X, Y)^n=\prod_{p\in \mathbb{Z}} {\rm Hom}_\mathfrak{a}(X^p, Y^{n+p}).$$
An element of ${\rm Hom}_\mathfrak{a}(X, Y)^n$ is denoted by $f=(f^p)_{p\in \mathbb{Z}}$, which might be viewed as a graded morphism of degree $n$ from $X$ to $Y$. The differential
$$d^n\colon {\rm Hom}_\mathfrak{a}(X, Y)^n\longrightarrow {\rm Hom}_\mathfrak{a}(X, Y)^{n+1}$$
sends $f$ to $d^nf$ such that $(d^nf)^p=d_Y^{p+n}\circ f^p-(-1)^n f^{p+1}\circ d_X^{p}$ for each $p\in \mathbb{Z}$. The importance of the Hom-complex lies in the following standard identities
\begin{align}\label{equ:Hom-complex}
Z^0{\rm Hom}_\mathfrak{a}(X, Y)={\rm Hom}_{C(\mathfrak{a})}(X, Y), \; \mbox{and } H^0{\rm Hom}_\mathfrak{a}(X, Y)={\rm Hom}_{\mathbf{K}(\mathfrak{a})}(X, Y).
\end{align}

We denote by $C_{\rm dg}(\mathfrak{a})$ the dg category of complexes, whose objects are just complexes in $\mathfrak{a}$ and whose morphism  complexes between $X$ and $Y$ are given by ${\rm Hom}_\mathfrak{a}(X, Y)$. The composition of morphisms in  $C_{\rm dg}(\mathfrak{a})$  is given by the composition of graded morphisms. In view of (\ref{equ:Hom-complex}), we infer that
$$Z^0(C_{\rm dg}(\mathfrak{a}))=C(\mathfrak{a}) \mbox{ and } H^0(C_{\rm dg}(\mathfrak{a}))=\mathbf{K}(\mathfrak{a}).$$ Therefore, the categories $C(\mathfrak{a})$ and $\mathbf{K}(\mathfrak{a})$ are shadows of the dg category $C_{\rm dg}(\mathfrak{a})$.

We denote by $C^{-}_{\rm dg}(\mathfrak{a})$, $C^{+}_{\rm dg}(\mathfrak{a})$ and $C^b_{\rm dg}(\mathfrak{a})$ the full dg subcategories of $C_{\rm dg}(\mathfrak{a})$ formed by bounded-above complexes, bounded-below complexes and bounded complexes, respectively.
\end{exm}

A dg functor $F\colon \mathcal{C}\rightarrow \mathcal{D}$ is said to be quasi-fully faithful if for any objects $X$ and $Y$ in $\mathcal{C}$, the cochain map
$$\mathcal{C}(X, Y)\longrightarrow \mathcal{D}(FX, FY), \; f\mapsto F(f),$$
is a quasi-isomorphism. In this case, the induced functor $H^0(F)\colon H^0(\mathcal{C})\rightarrow H^0(\mathcal{D})$ is fully faithful. If in addition $H^0(F)$ is dense, we call $F$ a \emph{quasi-equivalence}.

Following \cite[p.105, Remark]{BK2} and \cite[p.650]{Dri}, a dg category $\mathcal{C}$ is called \emph{strongly pretriangulated}, if it has internal shifts of objects and internal cones of closed morphisms of degree zero.

Let us explain the internal cones. For a closed morphism $f\colon X\rightarrow Y$ of degree zero, its \emph{internal cone} means an object $C$ which fits into a diagram in $\mathcal{C}$
    \[
\xymatrix{Y \ar[rr]^-{j} && C \ar@/^1pc/@{.>}[ll]^-{t} \ar[rr]^-{p} &&  X \ar@/^1pc/@{.>}[ll]^-{s} }\]
with $|j|=0=|t|$, $|p|=1$ and $|s|=-1$ subject to the following identities:
$$p\circ j=0=t\circ s,\ \ {\rm Id}_C=s\circ p+j\circ t,\ \  {\rm Id}_{Y}=t\circ j, \ \ {\rm Id}_{X}=p\circ s$$
and
$$ d_\mathcal{C}(j)=0=d_{\mathcal{C}}(p),\ \ d_\mathcal{C}(s)=j\circ f.$$
Such an internal cone is unique up to dg-isomorphism, if it exists. We mention that assuming the remaining assumptions, the condition $d_\mathcal{C}(s)=j\circ f$ is equivalent to  $d_\mathcal{C}(t)=-f\circ p$.

We mention that strongly pretriangulated dg categories are called exact in \cite{Kel99}.  In this case, the category $Z^0(\mathcal{C})$ has a canonical Frobenius exact structure whose stable category coincides with $H^0(\mathcal{C})$. Consequently, $H^0(\mathcal{C})$ has a canonical triangulated structure; see \cite[Subsection~2.1]{Kel99}.

It is well known that $C_{\rm dg}(\mathfrak{a})$ is strongly pretriangulated. The internal shifts of complexes are just usual shifts of complexes.  Since a closed morphism of degree zero in $C_{\rm dg}(\mathfrak{a})$  is precisely a cochain map, its internal
cone is given by the mapping cone. For the same reason,  $C^{-}_{\rm dg}(\mathfrak{a})$, $C^{+}_{\rm dg}(\mathfrak{a})$ and $C^b_{\rm dg}(\mathfrak{a})$ are all strongly pretriangulated.

One of the disadvantages of strongly pretriangulated categories is that they are not invariant under quasi-equivalences. A dg category $\mathcal{C}$ is \emph{pretriangulated} if it has shifts and cones up to homotopy; see \cite[p.650]{Dri}. In this case, the homotopy category $H^0(\mathcal{C})$ still has a canonical triangulated structure. If there is a quasi-equivalence $F\colon \mathcal{C}\rightarrow \mathcal{D}$, then $\mathcal{C}$ is pretriangulated if and only if so is $\mathcal{D}$, in which case $H^0(F)\colon H^0(\mathcal{C})\rightarrow H^0(\mathcal{D})$ is a triangle equivalence.

Let us explain   cones up to homotopy, which are somewhat subtle. For a closed morphism $f\colon X\rightarrow Y$ of degree zero, its \emph{cone up to homotopy} means an object $C'$ which fits into a diagram in $\mathcal{C}$
    \[
\xymatrix{Y \ar[rr]^-{j} && C' \ar@/^1pc/@{.>}[ll]^-{t} \ar[rr]^-{p} &&  X \ar@/^1pc/@{.>}[ll]^-{s} }\]
with $|j|=0=|t|$, $|p|=1$ and $|s|=-1$ subject to the following conditions:
\begin{enumerate}
\item $d_\mathcal{C}(j)=0=d_{\mathcal{C}}(p),\ \ d_\mathcal{C}(s)=j\circ f,\ \ d_\mathcal{C}(t)=-f\circ p$;
\item there exists $\varepsilon'\in \mathcal{C}(C', C')^{-1}$ such that ${\rm Id}_{C'}-d_\mathcal{C}(\varepsilon')=s\circ p+j\circ t$;
\item there exist morphisms $h\in \mathcal{C}(Y, X)^{0}$, $\varepsilon_X\in \mathcal{C}(X, X)^{-1}$, $\varepsilon_Y\in \mathcal{C}(Y, Y)^{-1}$,  and $r\in \mathcal{C}(X, Y)^{-2}$ satisfying the following identities:
\begin{align*}
    &p\circ j=d_\mathcal{C}(h), \quad    t\circ s=d_\mathcal{C}(r)+f\circ \varepsilon_X-\varepsilon_Y\circ f,\\
  & {\rm Id}_Y-t\circ j=d_\mathcal{C}(\varepsilon_Y)+f\circ h, \mbox{ and } {\rm Id}_X-p\circ s=d_\mathcal{C}(\varepsilon_X)+h\circ f.
\end{align*}
\end{enumerate}
Such a cone is unique up to homotopy equivalence, if it exists.

\begin{rem}
    Denote by ${\rm DGMod}\mbox{-}\mathcal{C}$ the dg category formed by right dg $\mathcal{C}$-modules. Consider the Yoneda dg functor
    $$\mathbf{Y}\colon \mathcal{C}\longrightarrow {\rm DGMod}\mbox{-}\mathcal{C},\quad  X\mapsto \mathcal{C}(-, X).$$
Assume that $f\colon X\rightarrow Y $ is a closed morphism  of degree zero in $\mathcal{C}$. Then an object $E$ in $\mathcal{C}$ is dg-isomorphic to the  internal cone of $f$ if and only if $\mathbf{Y}(E)$ is dg-isomorphic to the mapping cone of $\mathbf{Y}(f)$ in ${\rm DGMod}\mbox{-}\mathcal{C}$. Similarly, an object $E'$ in $\mathcal{C}$ is homotopy equivalent to the cone up to homotopy  of $f$ if and only if $\mathbf{Y}(E')$ is homotopy equivalent to the mapping cone of $\mathbf{Y}(f)$ in ${\rm DGMod}\mbox{-}\mathcal{C}$.
\end{rem}

The ``only if" part of the following well-known result is trivial, and the ``if" part is useful; see \cite[Lemma~3.1]{CC}.

\begin{lem}\label{lem:CC}
Let $F\colon \mathcal{C}\rightarrow \mathcal{D}$ be a dg functor between two pretriangulated dg categories. Then $F$ is a quasi-equivalence if and only if $H^0(F)\colon H^0(\mathcal{C})\rightarrow H^0(\mathcal{D})$ is a triangle equivalence.
\end{lem}

 Let $\mathcal{C}$ be a dg category. A \emph{dg ideal} $\mathcal{I}=\{\mathcal{I}(X, Y)\}_{X, Y\in {\rm Obj}(\mathcal{C})}$ consists of sub-complexes  $\mathcal{I}(X, Y)$ of $\mathcal{C}(X, Y)$ which satisfy the following condition: for any $f\in \mathcal{I}(X, Y)$, $a\in \mathcal{C}(X', X)$ and $b\in \mathcal{C}(Y, Y')$, the composition $b\circ f\circ a$ lies in $\mathcal{I}(X', Y')$. We have the factor dg category $\mathcal{C}/\mathcal{I}$, which has the same objects with $\mathcal{C}$ and whose morphism complexes are given by the quotient complex $\mathcal{C}(X, Y)/\mathcal{I}(X, Y)$. For a morphism $f$ in $\mathcal{C}$, the corresponding morphism in $\mathcal{C}/\mathcal{I}$ is denoted by $\underline{f}$.

 \begin{lem}\label{lem:factor}
 Let $\mathcal{C}$ be a strongly pretriangulated dg category and $\mathcal{I}$ a dg ideal of $\mathcal{C}$. Assume that for any morphism $f\colon X\rightarrow Y$ of degree zero satisfying $d_\mathcal{C}(f)\in \mathcal{I}(X, Y)$, there exist two morphisms $\iota\colon X'\rightarrow X$ and $\kappa \colon Y\rightarrow Y'$ of degree zero such that $d_\mathcal{C}(\kappa\circ f\circ \iota)=0$ and that both $\iota$ and $\kappa$ represent  dg-isomorphisms in $\mathcal{C}/\mathcal{I}$. Then $\mathcal{C}/\mathcal{I}$ is also strongly pretriangulated.
 \end{lem}

 \begin{proof}
Recall that a dg category $\mathcal{C}$ has internal shifts if and only if for each object $X$, there are a closed isomorphism $X\rightarrow X_1$ of degree $1$ and a closed isomorphism $X\rightarrow X_2$ of degree $-1$. Since the given dg category $\mathcal{C}$  has internal shifts, so does  the factor dg category $\mathcal{C}/\mathcal{I}$.

For the existence of internal cones, we observe that any closed morphism of degree zero in $\mathcal{C}/\mathcal{I}$ is of the form $\underline{f}$ for some morphism $f\colon X\rightarrow Y$ of degree zero in $\mathcal{C}$ satisfying $d_\mathcal{C}(f)\in \mathcal{I}(X, Y)$. By the assumption, we have morphisms $\iota\colon X'\rightarrow X$ and $\kappa \colon Y\rightarrow Y'$ of degree zero such that $\kappa\circ f\circ \iota$ is closed and that $\underline{\iota}$ and $\underline{\kappa}$ are  dg-isomorphisms. Therefore, the internal cone ${\rm Cone}(\kappa\circ f\circ \iota)$ of $\kappa\circ f\circ \iota$ exists in $\mathcal{C}$, which also is the internal cone of $\underline{\kappa\circ f\circ \iota}$ in $\mathcal{C}/\mathcal{I}$. As internal cones are invariant under dg-isomorphisms and  both $\underline{\iota}$ and $\underline{\kappa}$ are dg-isomorphisms in $\mathcal{C}/\mathcal{I}$, it follows that ${\rm Cone}(\kappa\circ f\circ \iota)$ is dg-isomorphic to the desired cone of $\underline{f}$ in $\mathcal{C}/\mathcal{I}$. This completes the proof.
 \end{proof}

 Denote by ${\bf dgcat}$ the category of small dg categories. The \emph{homotopy category} ${\bf Hodgcat}$ is obtained from ${\bf dgcat}$ by formally inverting quasi-equivalences. By the Dwyer-Kan model structure \cite{Tab} on ${\bf dgcat}$, the homotopy category ${\bf Hodgcat}$ has small Hom sets. Morphisms in ${\bf Hodgcat}$ are called dg quasi-functors. Here, the first two letters ``${\bf Ho}$" stand for the homotopy category in the sense of Quillen \cite{Qui}.

Assume that $\mathcal{C}$ is a small dg category. Let $\mathcal{N}$ be a full dg subcategory of $\mathcal{C}$. We denote by $\mathcal{C}/\mathcal{N}$ the \emph{dg quotient category} which is constructed as follows: first we take a semi-free resolution $\pi\colon \widetilde{\mathcal{C}}\rightarrow \mathcal{C}$ such that $\widetilde{\mathcal{C}}$ and $\mathcal{C}$ have the same objects and that $\pi$ acts on objects by the identity, see \cite[Appendix B, B.5~Lemma]{Dri}; secondly, we enlarge $\widetilde{\mathcal{C}}$ by freely adding an endomorphism $\epsilon_X$ of degree $-1$ for each object $X\in \mathcal{N}$, and set $d(\epsilon_X)={\rm Id}_X$. Here, by \emph{freely} adding $\epsilon_X$'s,  we form the dg tensor category of $\widetilde{\mathcal{C}}$ with respect to the free dg $\widetilde{\mathcal{C}}$-$\widetilde{\mathcal{C}}$-bimodule generated by these morphisms $\epsilon_X$'s, and then deform the differential of the dg tensor category by setting $d(\epsilon_X)={\rm Id}_X$. Therefore, the added endomorphism $\epsilon_X$ is a contracting homotopy for $X$. The resulting dg category is denoted by $\mathcal{C}/\mathcal{N}$. The morphism in  ${\bf Hodgcat}$  represented by the following roof
$$\mathcal{C}\stackrel{\pi}\longleftarrow \widetilde{\mathcal{C}}\stackrel{\rm inc}\longrightarrow \mathcal{C}/\mathcal{N}$$
is denoted by $q\colon \mathcal{C}\rightarrow \mathcal{C}/\mathcal{N}$, which is called the \emph{quotient dg quasi-functor}. For details on dg quotient categories, we refer to \cite[Section~4]{Kel99} and \cite[Subsection~3.1]{Dri}.

The following fundamental result is due to \cite[Theorem~3.4]{Dri}.

\begin{lem}\label{lem:dgquot}
Assume that both $\mathcal{C}$ and $\mathcal{N}$ are pretriangulated. Then so is the dg quotient category $\mathcal{C}/\mathcal{N}$, and the quotient dg quasi-functor $q$ induces a triangle equivalence
$$H^0(\mathcal{C})/H^0(\mathcal{N})\stackrel{\sim}\longrightarrow H^0(\mathcal{C}/\mathcal{N}). $$
\end{lem}

\begin{rem}
Since the induced triangle equivalence above acts on objects by the identity, it is really an isomorphism of triangulated categories.
\end{rem}

In view of the lemma above, the following definition is natural; see \cite{Kel18, BRTV, BrDy}.

\begin{defn}
Let $R$ be a ring. The \emph{dg singularity category} of $R$ is the dg quotient category
$$\mathbf{S}_{\rm dg}(R)=C_{\rm dg}^{-, b}(R\mbox{-proj})/{C_{\rm dg}^{b}(R\mbox{-proj})}.$$
\end{defn}

By Lemma~\ref{lem:dgquot}, $\mathbf{S}_{\rm dg}(R)$ is pretriangulated such that $H^0(\mathbf{S}_{\rm dg}(R))$ is isomorphic to the singularity category $\mathbf{D}_{\rm sg}(R)$ as a triangulated category.

\subsection{The Vogel dg category}

The following notion is taken from \cite[Subsection~1.3]{AV}. Let $\mathfrak{a}$ be an additive category.

\begin{defn}
A morphism $f=(f^p)_{p\in \mathbb{Z}}\colon X\rightarrow Y$ in  $C_{\rm dg}(\mathfrak{a})$  is called \emph{bounded} if $f^p=0$ whenever the absolute value of $p$ is sufficiently large, or equivalently, if there are only finitely many $p$'s with nonzero $f^p$.
\end{defn}

We observe that all bounded morphisms from $X$ to $Y$ form a sub-complex $\overline{\rm Hom}_\mathfrak{a}(X, Y)$ of ${\rm Hom}_\mathfrak{a}(X, Y)$. The corresponding quotient complex is denoted by $\widehat{\rm Hom}_\mathfrak{a}(X, Y)$.

To describe $H^0\widehat{\rm Hom}_\mathfrak{a}(X, Y)$, we recall the following terminologies in \cite[Section~2]{BC}. A morphism $f=(f^p)_{p\in \mathbb{Z}}\colon X\rightarrow Y$ of degree zero is called an \emph{almost-cochain map} if the morphism $d^0f$ is bounded. In other words, $d_Y^p\circ f^p=f^p\circ d_X^p$ holds for almost all $p$. A morphism $f$ is called \emph{almost null-homotopic} if there is a morphism $h=(h^p)_{p\in \mathbb{Z}}\colon X\rightarrow Y$ of degree $-1$ such that $f-d^{-1}h$ is bounded. This is equivalent to the condition that $f^p=d_Y^{p-1}\circ h^p+h^{p+1}\circ d_X^p$ holds for almost all $p$. It is clear that an almost null-homotopic morphism is an almost-cochain map. We observe that
$$H^0\widehat{\rm Hom}_\mathfrak{a}(X, Y)=\frac{\{\text{almost-cochain maps}\}}{\{ \text{almost null-homotopic morphisms}\}}.$$

The idea of the Vogel dg category goes back to Vogel's work, which appears in \cite{Goi}. For a historical account, we refer to \cite[1.4.2]{AV}.

\begin{defn}\label{defn:V}
The \emph{Vogel dg category} $\mathcal{V}(\mathfrak{a})$ is defined as follows. Its objects are complexes in $\mathfrak{a}$, and the morphism complex from $X$ to $Y$ is given by  $\widehat{\rm Hom}_\mathfrak{a}(X, Y)$. The composition is induced by the one in $C_{\rm dg}(\mathfrak{a})$.
\end{defn}

We mention that the composition in $\mathcal{V}(\mathfrak{a})$ is well defined, because ${\overline {\rm Hom}}_\mathfrak{a}(-, -)$ is a dg ideal of $C_{\rm dg}(\mathfrak{a})$. In other words, the Vogel dg category is the corresponding dg factor category  of $C_{\rm dg}(\mathfrak{a})$. We denote by $\mathcal{V}^-(\mathfrak{a})$ and $\mathcal{V}^+(\mathfrak{a})$ the full dg subcategory of $\mathcal{V}(\mathfrak{a})$ formed by bounded-above and bounded-below complexes, respectively.

\begin{prop}\label{prop:spretri}
The dg categories $\mathcal{V}(\mathfrak{a})$, $\mathcal{V}^{-}(\mathfrak{a})$ and $\mathcal{V}^{+}(\mathfrak{a})$ are all strongly pretriangulated.
\end{prop}

\begin{proof}
We only prove that $\mathcal{V}(\mathfrak{a})$ is strongly pretriangulated, since the other cases are actually consequences of the proof presented below.

Since $C_{\rm dg}(\mathfrak{a})$ is strongly pretriangulated, it suffices to prove that the dg ideal ${\overline {\rm Hom}}_\mathfrak{a}(-, -)$ satisfies the condition in Lemma~\ref{lem:factor}. Let $f\colon X\rightarrow Y$ be a morphism of degree zero such that $d^0f$ is bounded. Then there exists a sufficiently large natural number $n_0$ such that $d_Y^p\circ f^p=f^{p+1}\circ d_X^p$ whenever $p\geq n_0$ or $p\leq -n_0$.

Consider the direct sum $X'=\sigma_{\geq n_0}(X)\oplus \sigma_{\leq -n_0}(X)$ and $Y'=\sigma_{\geq n_0}(Y)\oplus \sigma_{\leq -n_0}(Y)$ of truncated complexes. We observe that both the obvious inclusion $\iota \colon X'\rightarrow X$ and projection $\kappa\colon Y\rightarrow Y'$ represent dg-isomorphisms in $\mathcal{V}$, whose inverses are given by the corresponding projection and inclusion, respectively. The composition $\kappa\circ f\circ\iota $ is a cochain map, that is, a closed morphism of degree zero in $C_{\rm dg}(\mathfrak{a})$. This verifies the condition in Lemma~\ref{lem:factor}.
\end{proof}

We have the following immediate consequence of the proof above.

\begin{cor}\label{cor:V}
For any complex $X$, the inclusion $\sigma_{\geq 1}(X)\oplus \sigma_{\leq -1}(X)\rightarrow X$ is a dg-isomorphism in $\mathcal{V}(\mathfrak{a})$. \hfill $\square$
\end{cor}

\begin{lem}\label{lem:Ext}
Let $X$ and $Y$ be two complexes. Then the following two statements hold.
\begin{enumerate}
\item Assume that $Y$ is bounded-above. Then we have an isomorphism of complexes
$$\widehat{\rm Hom}_\mathfrak{a}(X, Y)\simeq \colim\limits_{n\rightarrow +\infty}  {\rm Hom}_\mathfrak{a}(X, \sigma_{\leq -n}(Y)),$$
where the structure maps in the colimit are induced by the projections $\sigma_{\leq -n}(Y)\rightarrow \sigma_{\leq -(n+1)}(Y)$.
\item Assume that $X$ is bounded-below. Then we have an isomorphism of complexes
$$\widehat{\rm Hom}_\mathfrak{a}(X, Y)\simeq \colim\limits_{n\rightarrow +\infty} {\rm Hom}_\mathfrak{a}(\sigma_{\geq n}(X), Y),$$
where the structure maps are induced by the inclusions $\sigma_{\geq n+1}(X)\rightarrow \sigma_{\geq n}(X)$.
\end{enumerate}
\end{lem}

\begin{proof}
We only prove (1) since the proof of (2) is analogous. For each $n$, the inclusion  $\sigma_{> -n}(Y)\rightarrow Y$ and the projection $Y\rightarrow \sigma_{\leq -n}(Y)$ induce a short exact sequence of complexes.
\begin{align}\label{equ:Hom-ses}
0\longrightarrow {\rm Hom}_\mathfrak{a}(X, \sigma_{> -n}(Y))\longrightarrow {\rm Hom}_\mathfrak{a}(X, Y) \longrightarrow {\rm Hom}_\mathfrak{a}(X, \sigma_{\leq -n}(Y))\longrightarrow 0
\end{align}
By the assumption, $\sigma_{> -n}(Y)$ is bounded. Then we have
$${\rm Hom}_\mathfrak{a}(X, \sigma_{> -n}(Y))\subseteq \overline{\rm Hom}_\mathfrak{a}(X, Y).$$
Moreover, $\overline{\rm Hom}_\mathfrak{a}(X, Y)$ is identified with $\colim\limits_{n\rightarrow +\infty} {\rm Hom}_\mathfrak{a}(X, \sigma_{>-n}(Y))$, which is given by a direct union.

Letting $n$ vary in (\ref{equ:Hom-ses}) and taking colimits, we obtain  the following exact sequence.
\begin{align*}
0\longrightarrow \overline{\rm Hom}_\mathfrak{a}(X, Y)\longrightarrow {\rm Hom}_\mathfrak{a}(X, Y) \longrightarrow \colim\limits_{n\rightarrow +\infty} {\rm Hom}_\mathfrak{a}(X, \sigma_{\leq -n}(Y)) \longrightarrow 0
\end{align*}
By the very definition of $\widehat{\rm Hom}_\mathfrak{a}(X, Y)$, we infer the required isomorphism in (1).
\end{proof}

\begin{thm}\label{thm:Vogel}
Keep the notation as above. Then the following statements hold.
\begin{enumerate}
\item We have an orthogonal decomposition $H^0(\mathcal{V}(\mathfrak{a}))=H^0(\mathcal{V}^{-}(\mathfrak{a}))\times H^0(\mathcal{V}^{+}(\mathfrak{a}))$.
\item There are isomorphisms of triangulated categories:
$$\mathbf{K}^{-}(\mathfrak{a})/{\mathbf{K}^{b}(\mathfrak{a})}\stackrel{\sim}\longrightarrow H^0(\mathcal{V}^{-}(\mathfrak{a})) \quad \text{and} \quad \mathbf{K}^{+}(\mathfrak{a})/{\mathbf{K}^{b}(\mathfrak{a})}\stackrel{\sim}\longrightarrow H^0(\mathcal{V}^{+}(\mathfrak{a})).$$
\end{enumerate}
\end{thm}

\begin{proof} We write $\mathcal{V}$ for $\mathcal{V}(\mathfrak{a})$ in this proof.

(1) We take any bounded-below complex $X$ and any bounded-above complex $Y$. Applying $H^0$ to the isomorphism in Lemma~\ref{lem:Ext}(1), we obtain
\begin{align*}
{\rm Hom}_{H^0(\mathcal{V})}(X, Y) &\simeq H^0(\colim\limits_{n\rightarrow +\infty}  {\rm Hom}_\mathfrak{a}(X, \sigma_{\leq -n}(Y)))\\
& \simeq \colim \limits_{n\rightarrow +\infty}  H^0({\rm Hom}_\mathfrak{a}(X, \sigma_{\leq -n}(Y)))\\
&=  \colim\limits_{n\rightarrow +\infty} {\rm Hom}_{\mathbf{K}(\mathfrak{a})}(X, \sigma_{\leq -n}(Y)).
\end{align*}
However, when $n$ is sufficiently large, the complexes $X$ and $\sigma_{\leq -n}(Y)$ have disjoint supports and thus any cochain map between them is zero. Therefore, we have
$${\rm Hom}_{H^0(\mathcal{V})}(X, Y)=0.$$
Since the internal cones of $\mathcal{V}$ are inherited from the ones in $C_{\rm dg}(\mathfrak{a})$, (\ref{tri:can}) is still an exact  triangle in $H^0(\mathcal{V})$. Then we obtain a semi-orthogonal decomposition $H^0(\mathcal{V})=\langle H^0(\mathcal{V}^{-}),  H^0(\mathcal{V}^{+})\rangle$.

By Corollary~\ref{cor:V},  we have an isomorphism
$$X\simeq \sigma_{\geq 1}(X)\oplus \sigma_{\leq -1}(X)$$
in $H^0(\mathcal{V})$. In view of Remark~\ref{rem:ortho}, we infer that the semi-orthogonal decomposition $H^0(\mathcal{V})=\langle H^0(\mathcal{V}^{-}),  H^0(\mathcal{V}^{+})\rangle$ is actually orthogonal.

(2) We only prove the first of the two equivalences. The projection dg functor $C^{-}_{\rm dg}(\mathfrak{a})\rightarrow \mathcal{V}^-$ sends bounded complexes to contractible objects. Therefore, the induced triangle functor $\mathbf{K}^{-}(\mathfrak{a})\rightarrow H^0(\mathcal{V}^-)$ vanishes on $\mathbf{K}^b(\mathfrak{a})$. The latter further induces a triangle functor $\Phi\colon \mathbf{K}^{-}(\mathfrak{a})/{\mathbf{K}^b(\mathfrak{a})}\rightarrow H^0(\mathcal{V}^-)$. It is clear that $\Phi$ acts on objects by the identity. It suffices to prove that $\Phi$ is fully faithful.

Take two bounded-above complexes $U$ and $V$. By Lemma~\ref{lem:quot}(1), we have the first isomorphism in the following isomorphisms.
\begin{align*}{\rm Hom}_{\mathbf{K}(\mathfrak{a})/\mathbf{K}^b(\mathfrak{a})}(U, V) &\simeq \colim\limits_{n\rightarrow +\infty}  {\rm Hom}_{\mathbf{K}(\mathfrak{a})}(U, \sigma_{\leq -n}(V))\\
&= \colim\limits_{n\rightarrow +\infty} H^0({\rm Hom}_\mathfrak{a}(U, \sigma_{\leq -n}(V)))\\
&\simeq H^0(\colim\limits_{n\rightarrow +\infty} {\rm Hom}_\mathfrak{a}(U, \sigma_{\leq -n}(V)))\\
&\simeq H^0 \widehat{\rm Hom}_\mathfrak{a}(U, V)={\rm Hom}_{H^0(\mathcal{V}^{-})}(U, V).
\end{align*}
Here, the second isomorphism uses (\ref{equ:Hom-complex}) and the fourth one uses Lemma~\ref{lem:Ext}(1). It is routine to verify that this composite isomorphism is induced by the functor $\Phi$. This completes the proof.
\end{proof}

When the additive category $\mathfrak{a}$ is small, we easily enhance the results above to isomorphisms in $\mathbf{Hodgcat}$.

\begin{prop}\label{prop:V}
Assume that $\mathfrak{a}$ is small. Then there are isomorphisms in $\mathbf{Hodgcat}$:
$$C_{\rm dg}^{-}(\mathfrak{a})/{C_{\rm dg}^{b}(\mathfrak{a})} \times C_{\rm dg}^{+}(\mathfrak{a})/{C_{\rm dg}^{b}(\mathfrak{a})} \simeq \mathcal{V}^-(\mathfrak{a}) \times \mathcal{V}^+(\mathfrak{a})\simeq \mathcal{V}(\mathfrak{a}).$$
\end{prop}

\begin{proof}
We observe that the orthogonal decomposition in Theorem~\ref{thm:Vogel}(1) implies that the natural dg functor
$$\mathcal{V}^-(\mathfrak{a}) \times \mathcal{V}^+(\mathfrak{a})\longrightarrow \mathcal{V}(\mathfrak{a}),\; (X, Y)\mapsto X\oplus Y$$
induces a triangle equivalence $H^0(\mathcal{V}^-(\mathfrak{a})\times \mathcal{V}^+(\mathfrak{a}))\simeq H^0(\mathcal{V}(\mathfrak{a}))$. By Lemma~\ref{lem:dgquot}, we identify $H^0(C_{\rm dg}^{-}(\mathfrak{a})/{C_{\rm dg}^{b}(\mathfrak{a})})$ with $\mathbf{K}^-(\mathfrak{a})/{\mathbf{K}^b(\mathfrak{a})}$, and $H^0(C_{\rm dg}^{+}(\mathfrak{a})/{C_{\rm dg}^{b}(\mathfrak{a})})$ with $\mathbf{K}^+(\mathfrak{a})/{\mathbf{K}^b(\mathfrak{a})}$. Now for the required isomorphisms, we  just combine Lemma~\ref{lem:CC} with the results in Theorem~\ref{thm:Vogel}.
\end{proof}

\section{The singularity category}

In this section, we will study the Hom groups in the singularity category of a ring. We will recall Buchweitz's theorem \cite{Buc} on the singularity category of a Gorenstein ring, and report some progress on the Singular Presilting Conjecture \cite{CHQW}.

\subsection{Two long exact sequences}

Let $R$ be a ring. Recall that $\mathcal{V}(R\mbox{-proj})$ is the Vogel dg category of $R\mbox{-proj}$ as in Definition \ref{defn:V}. We consider its full dg subcategory $\mathcal{V}^{-, b}(R\mbox{-proj})$ consisting of bounded-above complexes with bounded cohomologies.

We relate the Vogel dg category to the dg singularity category.

\begin{prop}\label{prop:sing-V}
There is an isomorphism
$$\mathbf{S}_{\rm dg}(R)\simeq \mathcal{V}^{-, b}(R\mbox{-{\rm proj}})$$
 in $\mathbf{Hodgcat}$, which induces an isomorphism $\mathbf{D}_{\rm sg}(R)\simeq H^0(\mathcal{V}^{-, b}(R\mbox{-{\rm proj}}))$ of triangulated categories.
\end{prop}

\begin{proof}
We apply the proof of Proposition~\ref{prop:V} to $\mathfrak{a}=R\mbox{-proj}$, and obtain an isomorphism
$$C_{\rm dg}^{-}(R\mbox{-proj})/{C_{\rm dg}^{b}(R\mbox{-proj})}\simeq \mathcal{V}^-(R\mbox{-proj})$$
in ${\bf Hodgcat}$. Since this isomorphism acts on objects on the identity, it restricts to the required isomorphism.
$$\mathbf{S}_{\rm dg}(R)=C_{\rm dg}^{-,b}(R\mbox{-proj})/{C_{\rm dg}^{b}(R\mbox{-proj})}\simeq \mathcal{V}^{-,b}(R\mbox{-proj})$$
The induced isomorphism of triangulated categories follows from Lemma~\ref{lem:dgquot}.
\end{proof}

The following consideration is similar to the one in Proposition~\ref{prop:semi-ortho2}.

\begin{rem}
Denote by $C_{\rm dg}^{+, B}(R\mbox{-proj})$ the full dg subcategory of $C_{\rm dg}^{+}(R\mbox{-proj})$ formed by complexes $P$ satisfying $H^{-n}(P^*)=0$ for $n\gg 0$. Here, $P^*={\rm Hom}_R(P, R)$ is the dual complex of $P$. Similarly, we have the full dg subcategory $\mathcal{V}^{+, B}(R\mbox{-proj})$ of $\mathcal{V}^{+}(R\mbox{-proj})$.

The duality ${\rm Hom}_{R^{\rm op}}(-, R)$ identifies $\mathbf{S}_{\rm dg}(R^{\rm op})^{\rm op}$ with $C_{\rm dg}^{+, B}(R\mbox{-proj})/{C_{\rm dg}^{b}(R\mbox{-proj})}$. The proof of Proposition~\ref{prop:V} yields an isomorphism
$$C_{\rm dg}^{+, B}(R\mbox{-proj})/{C_{\rm dg}^{b}(R\mbox{-proj})}\simeq \mathcal{V}^{+, B}(R\mbox{-proj}).$$
In summary, we have an isomorphism
$$\mathbf{S}_{\rm dg}(R^{\rm op})^{\rm op} \simeq \mathcal{V}^{+, B}(R\mbox{-proj})$$
in ${\bf Hodgcat}$. We mention that  the homotopy category
$$H^0(C_{\rm dg}^{+, B}(R\mbox{-proj}))=\mathbf{K}^+(R\mbox{-proj})\cap \mathbf{K}^B(R\mbox{-proj})$$
appears in \cite[Proposition~7.12]{Nee08}.
\end{rem}

For any two complexes $P$ and $Q$ in $\mathbf{K}^{-, b}(R\mbox{-proj})$, we consider the dual complex $P^*={\rm Hom}_R(P, R)$, which is a bounded-below complex of finitely generated projective right $R$-modules. For each integer $n$, $H^{-n}(P^*\otimes_R Q)$ is the $n$-th hyper-Tor group and is denoted by ${\rm Tor}^R_{n}(P^*, Q)$.

The following long exact sequence extends the one in \cite[Theorem~6.2.5 (3)]{Buc}.

\begin{thm}\label{thm:les-1}
For any two complexes $P$ and $Q$ in $\mathbf{K}^{-, b}(R\mbox{-}{\rm proj})$, there is a long exact sequence of abelian groups,
\[\xymatrix{
\cdots \ar[r] & {\rm Tor}^R_{-n}(P^*, Q) \ar[r] & {\rm Hom}_{\mathbf{K}^{-, b}(R\mbox{-}{\rm proj})}(P, \Sigma^n(Q))\ar[dll]\\
 {\rm Hom}_{\mathbf{D}_{\rm sg}(R)}(P, \Sigma^n(Q))\ar[r] &
{\rm Tor}^R_{-n-1}(P^*, Q) \ar[r] & \cdots
}\]
where the slanted arrow is induced by the quotient functor $\mathbf{K}^{-, b}(R\mbox{-{\rm proj}})\rightarrow \mathbf{D}_{\rm sg}(R)$.
\end{thm}

\begin{proof}
We observe that the following canonical map
$$P^*\otimes_R Q \longrightarrow \overline{\rm Hom}_R(P, Q), \; f\otimes x\mapsto (y\mapsto (-1)^{|x|\cdot |y|}f(y)x),$$
is an isomorphism of complexes. Therefore,  we have a short exact sequence of complexes.
$$0\longrightarrow P^*\otimes_R Q \longrightarrow {\rm Hom}_R(P, Q)\longrightarrow \widehat{\rm Hom}_R(P, Q)\longrightarrow 0$$
By (\ref{equ:Hom-complex}), we identify ${\rm Hom}_{\mathbf{K}^{-, b}(R\mbox{-}{\rm proj})}(P, \Sigma^n(Q))$ with $H^n({\rm Hom}_R(P, Q))$; by  Proposition~\ref{prop:sing-V}, we identify ${\rm Hom}_{\mathbf{D}_{\rm sg}(R)}(P, \Sigma^n(Q))$ with $H^n(\widehat{\rm Hom}_R(P, Q))$. Then the associated long exact sequence of the exact sequence above yields the required one.
\end{proof}

We assume that the ring $R$ is left coherent. We identity $\mathbf{D}^b(R\mbox{-mod})$ with $\mathbf{K}^{-, b}(R\mbox{-{\rm proj}})$. Therefore, we have the following identification.
$$\mathbf{D}_{\rm sg}(R)=\mathbf{D}^b(R\mbox{-mod})/{\mathbf{K}^{b}(R\mbox{-{\rm proj}})}$$
In particular, we will identify an $R$-module $M$ with its projective resolution, and then view $M$ as an object in $\mathbf{D}_{\rm sg}(R)$. For each $R$-module $M$, we write $M^\vee=\mathbb{R}{\rm Hom}_R(M, R)$, which is an object in $\mathbf{D}(R^{\rm op}\mbox{-Mod})$.

\begin{cor}\label{cor:les-2}
Let $R$ be a left coherent ring and $M, N\in R\mbox{-{\rm mod}}$. Then there is a long exact sequence
\[\xymatrix{
\cdots \ar[r] & {\rm Tor}^R_{-n}(M^\vee, N) \ar[r] & {\rm Ext}_R^n(M, N)\ar[dll]\\
 {\rm Hom}_{\mathbf{D}_{\rm sg}(R)}(M, \Sigma^n(N))\ar[r] &
{\rm Tor}^R_{-n-1}(M^\vee, N) \ar[r] & \cdots
}\]
where the slanted arrow is induced by the quotient functor $\mathbf{D}^b(R\mbox{-{\rm mod}})\rightarrow \mathbf{D}_{\rm sg}(R)$. Consequently, for each $n\leq -2$, we have an isomorphism
$${\rm Hom}_{\mathbf{D}_{\rm sg}(R)}(M, \Sigma^n(N)) \stackrel{\sim}\longrightarrow  {\rm Tor}^R_{-(n+1)}(M^\vee, N).$$
\end{cor}

Here, we recall that ${\rm Ext}_R^n(M, N)=0$ for $n<0$ and ${\rm Ext}_R^0(M, N)={\rm Hom}_R(M, N)$.

\begin{proof}
We replace $M$ and $N$ by their projective resolutions $P$ and $Q$, respectively. Then $M^\vee$ is identified by $P^*$, and for each integer $n$, ${\rm Ext}_R^n(M, N)$ is isomorphic to ${\rm Hom}_{\mathbf{K}^{-, b}(R\mbox{-}{\rm proj})}(P, \Sigma^n(Q))$. Then the result follows immediately from Theorem~\ref{thm:les-1}. For the consequence, we just use the fact that ${\rm Ext}_R^n(M, N)=0$ for $n\leq -1$.
\end{proof}

\subsection{Buchweitz's theorem}
In this subsection, we assume that $R$ is a left coherent ring. For $M, N\in R\mbox{-mod}$, we have the following canonical exact sequence,
$$M^*\otimes_R N\xrightarrow{{\phi}_{M, N}}{\rm Hom}_R(M, N)\stackrel{\rm can}\longrightarrow \underline{\rm Hom}_R(M, N)\longrightarrow 0$$
where $\phi_{M, N}$ sends $f\otimes y$ to the morphism $(x\mapsto f(x)y)$ and the map ``can" is the canonical projection. Here, $M^*={\rm Hom}_R(M, R)$ is the dual module of $M$,  and  $\underline{\rm Hom}_R(M, N)$ denotes the Hom group in the stable category $R\mbox{-\underline{mod}}$.

We consider the following full subcategory of $R\mbox{-mod}$.
$${^\perp R}=\{M\in R\mbox{-mod}\; |\; {\rm Ext}_R^n(M, R)=0 \mbox{ for any }n\geq 1\}$$

The following result is partly due to \cite[Proposition~1.21]{Orl}; see also \cite[Lemma~3.4]{IY}.

\begin{prop}\label{prop:les-3}
Let $M, N\in R\mbox{-{\rm mod}}$ with $M\in {^\perp R}$. Then we have isomorphisms
$${\rm Hom}_{\mathbf{D}_{\rm sg}(R)}(M, \Sigma^n(N))\simeq \left\{ \begin{aligned}
                                                                    &{\rm Tor}^R_{-(n+1)}(M^*, N), &n\leq -2;\\
                                                                    & {\rm Ker}(\phi_{M, N}),& n=-1;\\
                                                                     & \underline{\rm Hom}_R(M, N), &n=0;\\
                                                                     & {\rm Ext}^n_R(M, N), &n\geq 1. \end{aligned}\right.$$

\end{prop}

\begin{proof}
By the assumption on $M$, we identify $M^\vee$ with the dual module $M^*$. Then the isomorphisms for $n\leq -2$ are contained already in Corollary~\ref{cor:les-2}.  The cases where $n\geq 1$ follow from the long exact sequence in Corollary~\ref{cor:les-2}, since ${\rm Tor}^R_{-i}(M^*, N)=0$ for $i\geq 1$. We observe from the same long exact sequence the following one
$$0\rightarrow {\rm Hom}_{\mathbf{D}_{\rm sg}(R)}(M, \Sigma^{-1}(N))\rightarrow M^*\otimes_R N \stackrel{\phi}\rightarrow {\rm Hom}_R(M, N)\rightarrow {\rm Hom}_{\mathbf{D}_{\rm sg}(R)}(M, N)\rightarrow 0.$$
We observe that $\phi$ coincides with $\phi_{M, N}$. This yields the isomorphisms in the remaining cases.
\end{proof}

An unbounded complex $P$ of finitely generated projective $R$-modules is totally acyclic if it is acyclic and its dual $P^*={\rm Hom}_R(P, R)$ is also acyclic. A finitely presented $R$-modules $G$ is called \emph{Gorenstein projective} \cite{EJ1} if there exists a totally acyclic complex $P$ such that its first cocycle $Z^1(P)$ is isomorphic to $G$. We denote by $R\mbox{-Gproj}$ the full subcategory formed by Gorenstein projective modules. We mention that the study of Gorenstein projective modules goes back to \cite{ABr}.

We observe that $R\mbox{-proj}\subseteq R\mbox{-Gproj}\subseteq {^\perp R}$. The full subcategory $R\mbox{-Gproj}$ of $R\mbox{-mod}$ is closed under extensions. Therefore, it  naturally becomes an exact category. Moreover, it is a Frobenius exact category such that its projective-injective objects are precisely projective $R$-modules. Therefore, by \cite[Theorem~I.2.8]{Hap}, its stable category $R\mbox{-\underline{Gproj}}$ is a triangulated category.

The \emph{Gorenstein projective dimension} \cite{EJ1, Hol} of a module $M$ is denoted by ${\rm Gpd}_R(M)$. Recall that ${\rm Gpd}_R(M)\leq n$ if and only if there is an exact sequence
$$0\rightarrow G_n\rightarrow \cdots \rightarrow G_1\rightarrow G_0\rightarrow M\rightarrow 0$$
 in $R\mbox{-mod}$ with each $G_i$ Gorenstein projective. We mention that a module $M$ is Gorenstein projective if and only if ${\rm Gpd}_R(M)=0$.

 \begin{defn}
 Let $R$ be a left coherent ring. We say that $R$ is \emph{left G-regular} if each module in $R\mbox{-mod}$ has finite Gorenstein projective dimension. The ring $R$ is \emph{right G-regular}, if $R^{\rm op}$ is left G-regular.
 \end{defn}

The following well-known result might be deduced from \cite[Corollary~2.21]{Hol}.

 \begin{lem}
 Assume that $R$ is left G-regular. Then we have $R\mbox{-{\rm Gproj}}={^\perp R}$.
 \end{lem}

 Let $d\geq 0$. A two-sided noetherian ring $S$ is said to be \emph{$d$-Gorenstein} if the injective dimension of $S$ on each side is at most $d$. By \cite[Lemma~A]{Zaks}, the two injective dimensions coincide.  As indicated in the following result due to \cite[Theorem~1.4]{HH}, $d$-Gorenstein rings are both left and right $G$-regular.

 \begin{prop}\label{prop:Goren}
 Let $S$ be a two-sided noetherian ring and $d\geq 0$. Then the following conditions are equivalent.
 \begin{enumerate}
 \item The ring $R$ is $d$-Gorenstein.
 \item ${\rm Gpd}_R(M)\leq d$ for any $M\in R\mbox{-{\rm mod}}$.
 \item ${\rm Gpd}_{R^{\rm op}}(N)\leq d$ for any $N\in R^{\rm op}\mbox{-{\rm mod}}$.
 \end{enumerate}
 \end{prop}

The coherent analogue of $d$-Gorenstein rings is as follows. A ring $R$ is \emph{$d$-FC} \cite{Dam, DC} if it is two-sided coherent and ${\rm Ext}_R^{d+1}(X, R)=0={\rm Ext}_{R^{\rm op}}^{d+1}(Y, R)$ for any $X\in R\mbox{-mod}$ and $Y\in R^{\rm op}\mbox{-mod}$. Therefore, a $d$-FC ring is $d$-Gorenstein if and only if it is two-sided noetherian.

The following result is a coherent analogue of \cite[Theorem~3.2]{AM}, and  implies that a $d$-FC ring is both left G-regular and right G-regular.

\begin{lem}
Let $R$ be a two-sided coherent ring and $d\geq 0$. Then $R$ is $d$-FC if and only if ${\rm Gpd}_R(M)\leq d$ and ${\rm Gpd}_{R^{\rm op}}(N)\leq d$  for any $M\in R\mbox{-{\rm mod}}$ and any $N\in R^{\rm op}\mbox{-{\rm mod}}$.
\end{lem}

\begin{proof}
The same proof of  \cite[Theorem~3.2]{AM} works here. We mention that the ``if" part also follows from  a finite analogue of \cite[Theorem~2.20]{Hol}, and the ``only if" part might be viewed as  a coherent analogue of \cite[Lemm~4.2.2 (iv)]{Buc}.
\end{proof}

\begin{rem}
Let $R$ be a two-sided coherent ring and $d\geq 0$. In view of Proposition~\ref{prop:Goren}, we expect that the following result might be true: if ${\rm Gpd}_R(M)\leq d$ for any $M\in R\mbox{-mod}$, then $R$ is $d$-FC.
\end{rem}

We consider the following composite functor
 \begin{align}\label{equ:Phi}
 R\mbox{-Gproj} \hookrightarrow R\mbox{-mod} \stackrel{\rm can}\longrightarrow \mathbf{D}^b(R\mbox{-mod})\stackrel{q}\longrightarrow \mathbf{D}_{\rm sg}(R).
 \end{align}
Here, ``can" identifies a module with the corresponding stalk complex concentrated in degree zero, and $q$ denotes the quotient functor. Since the composite functor vanishes on $R\mbox{-proj}$, it induces a well-defined functor
$$\Phi_R\colon R\mbox{-\underline{Gproj}}\longrightarrow \mathbf{D}_{\rm sg}(R).$$

The following result is due to \cite[Theorem~4.4]{Buc} in a slightly generalized form; compare \cite[Theorem~2.1]{Ric} and \cite[Theorem~4.6]{Hap2}.

\begin{thm}[Buchweitz]\label{thm:Buc}
Let $R$ be a left coherent ring. Then $\Phi_R$ is a fully faithful triangle functor. Moreover, $\Phi_R$ is dense, if and only if it is dense  up to direct summands,  if and only if $R$ is left G-regular.
\end{thm}

\begin{proof}
We observe that the composite functor (\ref{equ:Phi}) is a $\partial$-functor, which vanishes on projective modules. By \cite[Lemma~2.5]{Chen11}, the induced functor $\Phi_R$ is a triangle functor. Its fully-faithfulness follows from the case $n=0$ of Proposition~\ref{prop:les-3}.  In particular, the essential image of  $\Phi_R$ is a triangulated subcategory of $\mathbf{D}_{\rm sg}(R)$.

For the second statement, it suffices to prove the following two claims: (1) if $R$ is left G-regular, then $\Phi_R$ is dense; (2) if $\Phi_R$ is dense up to direct summands, then $R$ is left G-regular.

 To prove (1), we assume that $R$ is left G-regular. For an object $P\in \mathbf{K}^{-, b}(R\mbox{-proj})$, we take $n$ sufficiently large and apply the isomorphism (\ref{iso:proj}). Since $G=Z^{-n}(P)$ is  Gorenstein projective, it follows that $P$ is isomorphic to $\Sigma^n\Phi_R(G)$. Since the essential image of  $\Phi_R$ is a triangulated subcategory of $\mathbf{D}_{\rm sg}(R)$, we infer that $P$ lies in the essential image and that $\Phi_R$ is dense.

To prove (2), we take an arbitrary finitely presented $R$-module $M$ and take its projective resolution $Q$. Since $M$ lies in the essential image of $\Phi_R$ up to direct summands, it follows that $Q$ is isomorphic to a direct summand of the projective resolution $Q'$ of a Gorenstein projective module in $\mathbf{K}^{-, b}(R\mbox{-proj})/\mathbf{K}^b(R\mbox{-proj})$. We assume an isomorphism $Q\oplus Q''\simeq Q'$. By Lemma~\ref{lem:iso-in-sg}, we infer that higher syzygies of $M$ are isomorphic to direct summands of  $Z^{-n}(Q')$. The latter modules are Gorenstein projective. Recall that Gorenstein projective modules are closed under direct summands and isomorphisms in $R\mbox{-\underline{mod}}$.  It follows that  higher syzygies of $M$ are Gorenstein projective. Therefore, $M$ has finite Gorenstein projective dimension. Consequently,  the ring $R$  is left G-regular.
\end{proof}

\begin{rem}
    The original result \cite[Theorem~4.4]{Buc} only deals with Gorenstein rings, whose proof relies on complete resolutions  and  is completely different from the above one. The treatment here is similar to the one in \cite{Orl, Chen11} and relies on Proposition~\ref{prop:les-3}. For another proof, we refer to \cite[Section~2]{IY2}.  The ``only if" part of the second statement above is implicit in \cite[Theorem~6.9(8)]{Bel}.
\end{rem}

\begin{rem}
(1) Assume that $R$ is a commutative local Gorenstein ring. Then a module is Gorenstein projective if and only if it is maximal Cohen-Macaulay; see \cite[Proposition~4.12]{ABr}. In  view of this fact, Gorenstein projective modules over a noncomutative Gorenstein ring are also called maximal Cohen-Macaulay in \cite{Buc}.

(2) Assume that $R$ is a hypersurface singularity, that is, there is a regular local ring $S$ and a nonzero element $f$ in the maximal ideal of $S$ such that $R\simeq S/(f)$. Then the stable category of maximal Cohen-Macaulay $R$-modules is triangle equivalent to the homotopy category $\mathbf{HMF}(S; f)$ of \emph{matrix factorizations} of $f$;  see \cite[Theorem~6.1]{Eis}. We mention that matrix factorizations appear in the study of D-branes \cite{Orl}.

(3) Following \cite{BJO},  the \emph{Gorenstein defect category} of $R$ is defined to be the Verdier quotient category
$$\mathbf{D}_{\rm def}(R)=\mathbf{D}_{\rm sg}(R)/{{\rm Im}\; \Phi_R}.$$
The second statement above is reformulated as follows: $\mathbf{D}_{\rm def}(R)=0$ if and only if $R$ is left G-regular; compare \cite[Theorem~4.2]{BJO}.
\end{rem}

\subsection{The singular presilting conjecture}

Throughout this subsection, we assume that $\Lambda$ is an artin algebra \cite{ARS}, that is, $\Lambda$ is a finitely generated module over its center $Z(\Lambda)$ which is a commutative artinian ring.

Recall that the finitistic dimension of $\Lambda$ is defined as
$${\rm fin.dim}(\Lambda)={\rm sup}\{{\rm pd}_\Lambda (M)\; |\; M \in \Lambda\mbox{-mod} \mbox{ with } {\rm pd}_\Lambda(M)<\infty\}.$$
The following conjecture  is very well known \cite{Bass}.

\vskip 3pt

\emph{Finitistic Dimension Conjecture}. Let $\Lambda$ be an artin algebra. Then  ${\rm fin.dim}(\Lambda)<\infty$.

\vskip 3pt
The following two conjectures are proposed in \cite{AR75}.

\vskip 3pt

 \emph{Generalized Nakayama Conjecture.} Let $S$ be a simple module over an artin algebra $\Lambda$. Then ${\rm Ext}_\Lambda^n(S, \Lambda)\neq 0$ for some $n\geq 0$.

\vskip 3pt

 \emph{Auslander-Reiten Conjecture. } For a non-projective module $M$ over any artin algebra $\Lambda$, we have ${\rm Ext}^n_\Lambda(M, M\oplus \Lambda)\neq 0$ for some $n\geq 1$.

\vskip 3pt
We refer to \cite[ Conjectures, p.409-410]{ARS} for more conjectures in the representation theory of artin algebras.

By \cite[(8.6)~Theorem]{Tach}, the Generalized Nakayama Conjecture holds for group algebras of finite $p$-groups. By \cite[Theorem~3.4.3]{Yam}, if the Finististic Dimension Conjecture holds for $\Lambda^{\rm op}$, then the Generalized Nakayama Conjecture holds for $\Lambda$. Moreover, we have the following result due to \cite[Theorem~1.1]{AR75}.

 \begin{prop}
 The Auslander-Reiten Conjecture holds for all artin algebras if and only if the Generalized Nakayama Conjecture holds for all artin algebras.
 \end{prop}

Let $\mathcal{T}$ be a triangulated category. Recall from \cite[Definition~2.1]{AI} that an object $M$ is called \emph{presilting} if ${\rm Hom}_\mathcal{T}(M, \Sigma^n(M))=0$ for any $n>0$. If in addition the smallest thick triangulated subcategory containing $M$ is $\mathcal{T}$ itself, then $M$ is called a \emph{silting object}. We mention that the study of silting objects goes back to \cite{KV}.

The following conjecture is  proposed in \cite[Section~1]{CHQW}, which is inspired by \cite[Section~6]{IY}.

\vskip 3pt

 \emph{Singular Presilting Conjecture.} For any artin algebra $\Lambda$, there is no nonzero presilting object in $\mathbf{D}_{\rm sg}(\Lambda)$.

\vskip 3pt

The following result justifies the conjecture to some extent, and  is mentioned in \cite[Section~1]{CHQW} without a detailed proof.

\begin{prop}\label{prop:SPC-ARC}
Let $\Lambda$ be an artin algebra. If the Singular Presilting Conjecture holds for $\Lambda$, then so does the Auslander-Reiten Conjecture. The converse holds if $\Lambda$ is Gorenstein.
\end{prop}

\begin{proof}
We assume that the Singular Presilting Conjecture holds for $\Lambda$. Assume that $M$ is a $\Lambda$-module satisfying ${\rm Ext}^n_\Lambda(M, M\oplus \Lambda)=0$ for $n\geq 1$. Then $M$ belongs to ${^\perp \Lambda}$. By Proposition~\ref{prop:les-3}, we infer that ${\rm Hom}_{\mathbf{D}_{\rm sg}(\Lambda)}(M, \Sigma^n(M))=0$ for $n\geq 1$. Therefore, $M$ is a presilting object in $\mathbf{D}_{\rm sg}(\Lambda)$. By the assumption, we infer that $M$ is zero in $\mathbf{D}_{\rm sg}(\Lambda)$. In other words, $M$ has finite projective dimension. The following fact is standard: if ${\rm pd}_\Lambda (M)=d\geq 1$, we have ${\rm Ext}_\Lambda^d(M, \Lambda)\neq 0$. It follows that $M$ is projective. This implies that the Auslander-Reiten Conjecture holds for $\Lambda$.

For the converse, we assume that $\Lambda$ is Gorenstein and that the Auslander-Reiten Conjecture holds for $\Lambda$. Take a presilting object $X$ in $\mathbf{D}_{\rm sg}(\Lambda)$. By Theorem~\ref{thm:Buc}, we may assume that $X$ is given by a Gorenstein projective $\Lambda$-module, still denoted by $X$. By the case $n>0$ of Proposition~\ref{prop:les-3}, we infer that ${\rm Ext}_\Lambda^n(X, X)=0$ for $n\geq 1$. Since $X$ is Gorenstein projective, we have $X\in {^{\perp} \Lambda}$, that is, ${\rm Ext}_\Lambda^n(X, \Lambda)=0$ for $n\geq 1$. By the Ausalnder-Reiten Conjecture for $\Lambda$, we infer that $X$ is projective. Therefore, it is zero in $\mathbf{D}_{\rm sg}(\Lambda)$, as required.
\end{proof}

The following non-existence result generalizes \cite[Theorem~1]{AHMW} and partially supports the Singular Presilting Conjecture.

\begin{prop}{\rm (\cite[Corollary~3.3]{CLZZ})}
For any artin algebra $\Lambda$, there is no nonzero silting object in $\mathbf{D}_{\rm sg}(\Lambda)$.
\end{prop}

Recall that a $\Lambda$-module $M$ is ultimately-closed \cite[Section~3]{Jans} if there exist $d\geq 1$ such that $\Omega^d(M)$ belongs to ${\rm add}(\Lambda\oplus M\oplus \Omega(M)\oplus \cdots \oplus \Omega^{d-1}(M))$. Here, $\Omega(M)$ denotes the first syzygy of $M$,  and ${\rm add}(X)$ means the full subcategory formed by direct summands of finite direct sums of a module $X$.  An artin algebra $\Lambda$ is called \emph{ultimately-closed} if each $\Lambda$-module is ultimately-closed. For example, any syzygy-finite algebra is ultimately-closed.

By \cite[Proposition~1.3]{AR75}, the Auslander-Reiten  Conjecture holds for any ultimately-closed algebra. In view of Proposition~\ref{prop:SPC-ARC}, the following result is expected.

\begin{prop}{\rm (\cite[Proposition~3.5]{CLZZ})}
The Singular Presilting Conjecture holds for any ultimately-closed algebra.
\end{prop}

\section{The singular Yoneda dg category}

In this section, we recall the singular Yoneda dg category from \cite{CW}, which provides another explicit dg enhancement for the singularity category; see Proposition~\ref{prop:sing-SY}.

We fix a semisimple artinian ring $E$ and a ring homorphism $E\rightarrow \Lambda$.  We will abbreviate $\otimes_E$ as $\otimes$.

\subsection{The bar and Yoneda dg categories}

 Denote $\overline \Lambda= \Lambda/(E 1_\Lambda)$, which is an $E$-$E$-bimodule. Its $1$-shift stalk complex $s\overline \Lambda$ is concentrated in degree $-1$, whose typical element $s\overline{a}$ has degree $-1$.

The normalized $E$-relative bar resolution $\mathbb{B}$ of $\Lambda$ is a complex of $\Lambda$-$\Lambda$-bimodules given as follows. As a graded $\Lambda$-$\Lambda$-bimodule, we have
$$\mathbb{B}=\Lambda\otimes T_E  (s\overline{\Lambda})\otimes   \Lambda,$$
where ${\rm deg}(a_0\otimes   s\overline{a}_{1, n}\otimes   a_{n+1})=-n$. Here, for simplicity, we write
$$s\overline{a}_{1, n} := s\overline{a}_1\otimes   s\overline{a}_2 \otimes  \cdots \otimes   s\overline{a}_n.$$
The differential $d$ is given such that $d(a_0\otimes   a_1)=0$ and that
\begin{align*}
d(a_0\otimes   s\overline{a}_{1, n}\otimes   a_{n+1})={}  & a_0a_1\otimes   s\overline{a}_{2,n}\otimes   a_{n+1}+ (-1)^n a_0\otimes   s\overline{a}_{1,n-1}\otimes   a_{n}a_{n+1}\\
&+\sum_{i=1}^{n-1} (-1)^i a_0\otimes   s\overline{a}_{1, i-1}\otimes   s\overline{a_ia_{i+1}}\otimes   s\overline{a}_{i+2, n}\otimes   a_{n+1}.
\end{align*}
Here and later, as usual, $s\overline{a}_{1, 0}$ and $s\overline{a}_{n+1, n}$ are understood to be the empty word and should be ignored. We mention that bar resolutions are due to \cite[Chapter~II]{EM} and \cite[Chapter~IX, Section~6]{CE}.

It is well known that $\mathbb{B}$ is a coalgebra in the monoidal category of  complexes of $\Lambda$-$\Lambda$-bimodules. To be more precise, we have a cochain map between complexes of $\Lambda$-$\Lambda$-bimodules
$$\Delta\colon \mathbb{B}\longrightarrow \mathbb{B}\otimes_\Lambda\mathbb{B}$$
given by
$$\Delta(a_0\otimes   s\overline{a}_{1, n}\otimes   a_{n+1})=\sum_{i=0}^n (a_0\otimes   s\overline{a}_{1,i}\otimes   1_\Lambda)\otimes_\Lambda (1_\Lambda\otimes   s\overline{a}_{i+1, n}\otimes   a_{n+1}).$$
The natural cochain map $$\varepsilon\colon  \mathbb{B}\longrightarrow \Lambda$$ is  given by the projection $\mathbb{B}\rightarrow \Lambda \otimes \Lambda$ followed by the multiplication map of $\Lambda$. We have the following coassociative property
$$(\Delta\otimes_\Lambda {\rm Id}_\mathbb{B})\circ \Delta = ({\rm Id}_\mathbb{B} \otimes_\Lambda \Delta)\circ \Delta$$
and the counital property
$$(\varepsilon\otimes_\Lambda {\rm Id}_\mathbb{B}) \circ \Delta={\rm Id}_\mathbb{B}=({\rm Id}_\mathbb{B}\otimes_\Lambda \varepsilon) \circ \Delta.$$
Here, in the identity above, we identify both $\Lambda\otimes_\Lambda \mathbb{B}$ and $\mathbb{B}\otimes_\Lambda \Lambda$ with $\mathbb{B}$.

Following the treatment in \cite[Subsection~6.6]{Kel94}, we define the \emph{$E$-relative bar dg category} $\mathcal{B}=\mathcal{B}_{\Lambda/E}$ as follows. The objects are precisely all the complexes of $\Lambda$-modules, and the morphism complex between two objects $X$ and $Y$ is given by the Hom-complex
$$\mathcal{B}(X, Y)={\rm Hom}_\Lambda(\mathbb{B}\otimes_\Lambda X, Y).$$
Here, we refer to Example~\ref{exm:Hom-complex} for Hom-complexes. The composition $\ast$ of two morphisms $f\in \mathcal{B}(X, Y)$ and $g\in \mathcal{B}(Y, Z)$ is defined to be
$$g\ast f := (\mathbb{B}\otimes_\Lambda X\xrightarrow{\Delta\otimes_\Lambda {\rm Id}_X} \mathbb{B}\otimes_\Lambda \mathbb{B}\otimes_\Lambda X\xrightarrow{{\rm Id}_\mathbb{B}\otimes_\Lambda f} \mathbb{B}\otimes_\Lambda Y\xrightarrow{g} Z).$$
Moreover, the identity endomorphism in $\mathcal{B}(X, X)$ is given by
$$\mathbb{B}\otimes_\Lambda X \xrightarrow{\varepsilon\otimes_\Lambda {\rm Id}_X} \Lambda\otimes_\Lambda X=X.$$

In what follows, we will unpack the above definition of $\mathcal{B}$ and obtain its alternative form.

The \emph{$E$-relative Yoneda dg category} $\mathcal{Y}=\mathcal{Y}_{\Lambda/E}$ has the same objects as $\mathcal{B}$ has. For two complexes $X$ and $Y$ of $\Lambda$-modules, the underlying graded $\mathbb{K}$-module of the morphism complex $\mathcal{Y}(X, Y)$ is given by an infinite product
$$\mathcal{Y}(X, Y)=\prod_{n\geq 0}{\rm Hom}_E((s\overline{\Lambda})^{\otimes   n}\otimes  X, Y).$$
We denote
$$\mathcal{Y}_n(X, Y): ={\rm Hom}_E((s\overline{\Lambda})^{\otimes   n}\otimes   X, Y),$$
and say that elements in $\mathcal{Y}_n(X, Y)$ are of \emph{ filtration-degree} $n$. Observe that $\mathcal{Y}_0(X, Y)={\rm Hom}_E(X, Y)$.  The differential $\delta$ of $\mathcal{Y}(X, Y)$ is determined by
$$\begin{pmatrix}\delta_{\rm in}\\
\delta_{\rm ex}\end{pmatrix}\colon \mathcal{Y}_n(X, Y)\longrightarrow \mathcal{Y}_n(X, Y)\oplus \mathcal{Y}_{n+1}(X, Y),$$
where
$$\delta_{\rm in}(f)(s\overline{a}_{1,n}\otimes   x)=d_Y (f(s\overline{a}_{1,n}\otimes   x))-(-1)^{|f|+n}f(s\overline{a}_{1,n}\otimes   d_X(x))$$
and
\begin{align*}
\delta_{\rm ex}(f)(s\overline{a}_{1,n+1}\otimes   x) = & (-1)^{|f|+1}a_1 f(s\overline{a}_{2,n+1}\otimes   x)+(-1)^{|f|+n}f(s\overline{a}_{1,n}\otimes   a_{n+1}x)\\
&+\sum_{i=1}^{n}(-1)^{|f|+i+1} f(s\overline{a}_{1, i-1}\otimes   s\overline{a_ia_{i+1}}\otimes   s\overline{a}_{i+2, n+1}\otimes   x).
\end{align*}
We translate the composition $\ast$ in $\mathcal{B}$ into a cup-type product $\odot$ as follows; compare \cite[Section~7]{Gers}. The composition $\odot$ of morphisms in $\mathcal{Y}$ is defined such that, for $f\in \mathcal{Y}_n(X, Y)$ and $g\in \mathcal{Y}_m(Y, Z)$, their composition $g\odot f\in \mathcal{Y}_{n+m}(X, Z)$ is given by
\begin{equation*}
(g\odot f)(s\overline{a}_{1, m+n}\otimes   x)=(-1)^{m|f|} g(s\overline{a}_{1, m}\otimes   f(s\overline{a}_{m+1, m+n}\otimes   x)).
\end{equation*}
The identity endomorphism in $\mathcal{Y}(X, X)$ is given by the genuine identity map ${\rm Id}_X\in \mathcal{Y}_0(X, X)\subseteq \mathcal{Y}(X, X)$.

The canonical isomorphism $$\mathcal{B}(X, Y)\simeq \mathcal{Y}(X, Y)$$
of complexes sends $\tilde{f}\in {\rm Hom}_\Lambda((\Lambda\otimes (s\overline{\Lambda})^{\otimes n}\otimes \Lambda) \otimes_\Lambda X, Y)$ to $f\in \mathcal{Y}_n(X, Y)$, which is given by
$$f(s\overline{a}_{1, n}\otimes x)=\tilde{f}((1\otimes s\overline{a}_{1, n}\otimes 1)\otimes_\Lambda x).$$
The isomorphisms above for all $X$ and $Y$  yield the following result; see \cite[Lemma~7.1]{CW}.

\begin{lem}
There is an isomorphism $\mathcal{B}_{\Lambda/E}\simeq \mathcal{Y}_{\Lambda/E}$ of dg categories.
\end{lem}

Consider the natural dg functor
$${\rm \Theta}\colon C_{\rm dg}(\Lambda\mbox{-Mod})\longrightarrow \mathcal{Y}=\mathcal{Y}_{\Lambda/E}$$
which acts on objects by the identity, and identifies $f\in {\rm Hom}_\Lambda(X, Y)$ with $f\in {\rm Hom}_E(X, Y) = \mathcal{Y}_0(X, Y)\subseteq \mathcal{Y}(X, Y)$. Indeed, $C_{\rm dg}(\Lambda\mbox{-Mod})$ might be viewed as a non-full dg subcategory of $\mathcal{Y}$.

Denote by $\mathcal{Y}_{\Lambda/E}^f$ the full dg subcategory of $\mathcal{Y}_{\Lambda/E}$ formed by the complexes in $C^{-, b}(\Lambda\mbox{-proj})$.

The following result is  a finite version of  \cite[Proposition~7.3]{CW}; compare \cite[Corollary~7.5]{CW}.

\begin{prop}\label{prop:Theta}
The above dg functor $\Theta$ induces an isomorphism in $\mathbf{Hodgcat}$
 $$\Theta\colon C_{\rm dg}^{-,b}(\Lambda\mbox{-}{\rm proj})\simeq \mathcal{Y}^f_{\Lambda/E}.$$
 Consequently, $\mathcal{Y}^f_{\Lambda/E}$ is pretriangulated and we have an isomorphism
 $\mathbf{K}^{-, b}(\Lambda\mbox{-}{\rm proj})\simeq H^0(\mathcal{Y}^f_{\Lambda/E})$
 of triangulated categories.
\end{prop}

We mention that \cite[Proposition~7.3]{CW} justifies our terminology for $\mathcal{Y}_{\Lambda/E}$:  for each $\Lambda$-module $M$, the cohomology ring of the dg ring $\mathcal{Y}_{\Lambda/E}(M, M)$ is isomorphic to the \emph{Yoneda ring} of $M$:
$${\rm Ext}^*_{\Lambda}(M, M)=\bigoplus_{i\geq 0} {\rm Hom}_{\mathbf{D}(\Lambda\mbox{-}{\rm Mod})} (M, \Sigma^i(M)).$$

\subsection{Noncommutative differential forms}
We will study noncommutative differential forms with values in a complex; see \cite[Section~8]{CW}.  This gives rise to a dg endofunctor $\Omega_{\rm nc}$ on the Yoneda dg category $\mathcal Y=\mathcal{Y}_{\Lambda/E}$.  We will  obtain a closed natural transformation $\theta \colon {\rm Id}_{\mathcal Y} \to \Omega_{\rm nc}$ of degree zero satisfying $\theta \Omega_{\rm nc} = \Omega_{\rm nc} \theta$.

Let $X$ be a complex of $\Lambda$-modules. The complex of \emph{$X$-valued  noncommutative differential $1$-forms} is defined by
$$\Omega_{{\rm nc}, \Lambda/E}(X)=s\overline{\Lambda}\otimes X,$$
which is graded by means of ${\rm deg}(s\overline{a}_1\otimes   x)=|x|-1$ and whose differential is given by $d(s\overline{a}_1\otimes   x)=-s\overline{a}_1\otimes   d_X(x)$. The $\Lambda$-action on $\Omega_{{\rm nc}, \Lambda/E}(X)$ is given by the following nontrivial rule:
\begin{align}\label{equ:nontri-rule}
a\blacktriangleright (s\overline{a}_1\otimes   x)=s\overline{aa_1}\otimes   x-s\overline{a}\otimes   a_1x
\end{align}
for all $a\in \Lambda$.

To justify the  terminology above, we observe that
$$\Omega_{{\rm nc}, \Lambda/E}(\Lambda)=s\overline{\Lambda}\otimes   \Lambda$$
 is a stalk complex of $\Lambda$-$\Lambda$-bimodules concentrated in degree $-1$, where the right $\Lambda$-action is given by the multiplication of $\Lambda$. This stalk complex is called the \emph{graded bimodule of right noncommutative differential $1$-forms} \cite{Wan1}. Moreover, we have a canonical isomorphism
 $$\Omega_{{\rm nc}, \Lambda/E}(\Lambda)\otimes_\Lambda X\simeq \Omega_{{\rm nc}, \Lambda/E}(X), $$
 which sends $(s\overline{a}_0\otimes   a_1)\otimes_\Lambda x$ to $s\overline{a}_0\otimes   a_1x$. We mention that the study of noncommutative differential forms goes back to \cite[Sections~1~and~2]{CQ}.

 To avoid notational overload, we write $\Omega_{\rm nc}(X)=\Omega_{{\rm nc}, \Lambda/E}(X)$. We have a dg functor $$\Omega_{\rm nc}\colon \mathcal{Y}\longrightarrow \mathcal{Y}, \quad X\mapsto \Omega_{\rm nc}(X),$$
which sends a morphism $f\in \mathcal{Y}_n(X, Y)$ to the following morphism in $\mathcal{Y}_{n}(\Omega_{\rm nc}(X), \Omega_{\rm nc}(Y))$:
 $$  \xymatrix@C=3pc{
 (s\overline{\Lambda})^{\otimes   n}\otimes   \Omega_{\rm nc}(X) \ar[r]^-{=}  &  (s\overline{\Lambda})^{\otimes   (n+1)}\otimes   X \ar[r]^-{\mathrm{Id}_{s\overline{\Lambda}}\otimes   f} & s\overline{\Lambda}\otimes   Y=\Omega_{\rm nc}(Y).
 }$$

 We have a closed natural transformation of degree zero
 $$\theta\colon {\rm Id}_\mathcal{Y}\longrightarrow \Omega_{\rm nc}$$
defined as follows. For any $X \in \mathcal{Y}$, $\theta_X$ lies in $\mathcal{Y}_1(X, \Omega_{\rm nc}(X))\subseteq \mathcal{Y}(X, \Omega_{\rm nc}(X))$ and is given by
$$\theta_X(s\overline{a}\otimes   x)=s\overline{a}\otimes   x\in \Omega_{\rm nc}(X).$$
 Observe that $\theta_X$ is of degree zero and that $\delta(\theta_X)=0$ using the nontrivial rule (\ref{equ:nontri-rule}). Therefore, $\theta_X$ is a closed morphism of degree zero in $\mathcal{Y}$.  In order to prove that $\theta$ is natural,  we observe that for each $f\in \mathcal{Y}_n(X, Y)$, we have
$$\theta_Y\odot f=\Omega_{\rm nc}(f)\odot \theta_X.$$
Indeed, both sides send $s\overline{a}_{1, n+1}\otimes   x$ to $(-1)^{|f|}s\overline{a}_1\otimes   f(s\overline{a}_{2, n}\otimes   x)$.

 We observe that
 $$\Omega_{\rm nc}(\theta_X)=\theta_{\Omega_{\rm nc}(X)},$$
 since both sides lie in $\mathcal{Y}_1(\Omega_{\rm nc}(X), \Omega_{\rm nc}^2(X)) = \mathrm{Hom}  (\Omega_{\rm nc}^2(X), \Omega_{\rm nc}^2(X))$ and correspond to the identity map of $\Omega_{\rm nc}^2(X)$.

 We will see that both $\Omega_{\rm nc}$ and $\theta$ are closely related to the truncations of $\mathbb{B}$. For each $p\geq 0$, we consider the following subcomplex of $\mathbb{B}$:
$$\mathbb{B}_{<p}=\bigoplus_{0\leq n < p} \Lambda \otimes (s\bar{\Lambda})^{\otimes n}\otimes \Lambda.$$
Here, we understand $\mathbb{B}_{<0}$ as the zero complex, and $\mathbb{B}_{<1}$ as $\Lambda \otimes \Lambda=\Lambda \otimes E \otimes \Lambda$. The corresponding  quotient complex $\mathbb{B}/\mathbb{B}_{<p}$ will be denoted by $\mathbb{B}_{\geq p}$.

Let $X$ be a complex of $\Lambda$-modules and $p\geq 0$. We have a canonical  isomorphism
\begin{align}\label{equ:B-p}
\mathbb{B}_{\geq p}\otimes_\Lambda \Omega_{\rm nc}(X)\simeq \mathbb{B}_{\geq p+1}\otimes_\Lambda X
\end{align}
of complexes of $\Lambda$-modules, sending $(a_0\otimes s\bar{a}_{1, n}\otimes 1)\otimes_\Lambda (s\bar{a}_{n+1}\otimes x)$ to $(a_0\otimes s\bar{a}_{1, n+1}\otimes 1)\otimes_\Lambda x$ for $n\geq p$.

The quasi-isomorphism
$$\varepsilon\otimes {\rm Id}_X\colon \mathbb{B}\otimes_\Lambda X \longrightarrow X$$ is viewed as an element in $\mathcal{Y}_0(\mathbb{B}\otimes_\Lambda X, X)$, which is  a closed morphism  of degree zero in $\mathcal{Y}$. There is another closed morphism  of degree zero in $\mathcal{Y}$
\begin{align*}
\iota_X\colon X\longrightarrow \mathbb{B}\otimes_\Lambda X.
\end{align*}
For each $p\geq 0$, we define the entry $(\iota_X)_p\in \mathcal{Y}_p(X, \mathbb{B}\otimes_\Lambda X)$ by the following map:
$$(s\bar{\Lambda})^{\otimes p}\otimes X \longrightarrow \mathbb{B}^{-p}\otimes_\Lambda X\subseteq \mathbb{B}\otimes_\Lambda X, \quad s\bar{a}_{1, p}\otimes x\longmapsto (1\otimes s\bar{a}_{1, p}\otimes 1)\otimes_\Lambda x.$$
Then we set $\iota_X=((\iota_X)_p)_{p\geq 0}\in \mathcal{Y}(X, \mathbb{B}\otimes_\Lambda X)$. It is direct to verify the following identity  in $\mathcal{Y}$:
\begin{align*}
(\varepsilon\otimes_\Lambda {\rm Id}_X)\odot  \iota_X={\rm Id}_X.
\end{align*}
By \cite[Section~5]{CW2},  $\iota_X$ is an inverse of $\varepsilon\otimes_\Lambda {\rm Id}_X$ in $H^0(\mathcal{Y})$.

Denote by $\pi_0\colon \mathbb{B}\rightarrow \mathbb{B}_{\geq 1}=\mathbb{B}/ \mathbb{B}_{<1}$ the natural projection. The following lemma, due to \cite[Lemma~5.3]{CW2}, relates $\theta_X$ to  $\pi_0\otimes_\Lambda {\rm Id}_X$.

\begin{lem}\label{lem:comm-squ}
The following diagram
\[\xymatrix{
X \ar[r]^-{\theta_X} \ar[d]_-{\iota_X} & \Omega_{\rm nc}(X)\ar[d]^-{\iota_{\Omega_{\rm nc}(X)}}\\
\mathbb{B}\otimes_\Lambda X \ar[r] & \mathbb{B}\otimes_\Lambda \Omega_{\rm nc}(X)
}
\]
commutes in $\mathcal{Y}$, where the lower arrow is the composition of $\pi_0\otimes_\Lambda {\rm Id}_X$ with $\mathbb{B}_{\geq 1}\otimes_\Lambda X\rightarrow \mathbb{B}\otimes_\Lambda \Omega_{\rm nc}(X)$, the inverse of  the canonical isomorphism (\ref{equ:B-p}) with $p=0$.
\end{lem}

 \subsection{The singular Yoneda dg category and dg singularity category}

  Associated to the triple $(\mathcal{Y}, \Omega_{\rm nc}, \theta)$ and using \cite[Section~6]{CW}, we form  a strict dg localization along $\theta$:
$$\iota\colon \mathcal{Y}\longrightarrow \mathcal{SY}_{\Lambda/E}.$$
The obtained dg category $\mathcal{SY}=\mathcal{SY}_{\Lambda/E}$ is called the \emph{$E$-relative singular Yoneda dg category} of $\Lambda$.

Let us describe $\mathcal{SY}$ explicitly. Its objects are just complexes of $\Lambda$-modules. For two objects $X$ and $Y$, the morphism complex is defined to be the colimit of the following sequence of cochain complexes.
$$\mathcal{Y}(X, Y)\longrightarrow \mathcal{Y}(X, \Omega_{\rm nc}(Y))\longrightarrow \cdots \longrightarrow \mathcal{Y}(X, \Omega_{\rm nc}^p(Y))\longrightarrow \mathcal{Y}(X, \Omega_{\rm nc}^{p+1}(Y))\longrightarrow \cdots$$
The structure map sends $f$ to $\theta_{\Omega_{\rm nc}^p(Y)}\odot f$. More precisely, for any $f \in  \mathcal{Y}_n(X, \Omega_{\rm nc}^p(Y))$, the map $ \theta_{\Omega_{\rm nc}^p(Y)}\odot f \in  \mathcal{Y}_{n+1}(X, \Omega_{\rm nc}^{p+1}(Y))$ is given by
\begin{align*}
s \overline{a}_{1, n+1} \otimes   x \longmapsto (-1)^{|f|} s\overline{a}_1 \otimes   f(s\overline{a}_{2, n+1}\otimes   x).
\end{align*}

The image of $f\in \mathcal{Y}(X, \Omega_{\rm nc}^p(Y))$ in $\mathcal{SY}(X, Y)$ is denoted by $[f;p]$. The composition $\odot_{\rm sg} $ of $[f;p]$ with $[g;q]\in \mathcal{SY}(Y, Z)$ is defined by
$$[g;q]\odot_{\rm sg} [f;p]=[\Omega_{\rm nc}^p(g)\odot f;p+q].$$
Recall from Proposition~\ref{prop:Theta}  the Yoneda dg category $\mathcal{Y}$ is pretriangulated. By a general result \cite[Lemma~6.3]{CW},  the dg category $\mathcal{SY}$ is also pretriangulated.

We denote by $\mathcal{SY}^f_{\Lambda/E}$ the full dg subcategory of $\mathcal{SY}_{\Lambda/E}$ formed by the complexes in $C^{-, b}(\Lambda\mbox{-proj})$.

The following result is a finite version of \cite[Proposition~9.1]{CW}; compare \cite[Corollary~9.3]{CW}.

\begin{prop}\label{prop:sing-SY}
Keep the notation as above. Then the composite dg functor $C^{-,b}_{\rm dg}(\Lambda\mbox{-}{\rm proj}) \stackrel{\Theta} \rightarrow \mathcal{Y}_{\Lambda/E}\stackrel{\iota}\rightarrow \mathcal{SY}_{\Lambda/E}$ induces an isomorphism in $\mathbf{Hodgcat}$:
$$\mathbf{S}_{\rm dg}(\Lambda)\simeq \mathcal{SY}_{\Lambda/E}^f.$$
Consequently, we have an isomorphism  $\mathbf{D}_{\rm sg}(\Lambda)\simeq H^0(\mathcal{SY}_{\Lambda/E}^f)$ of triangulated categories.
\end{prop}

\begin{rem}
(1) We mention that \cite[Proposition~9.1]{CW} justifies our terminology for $\mathcal{SY}_{\Lambda/E}$:  for each $\Lambda$-module $M$, the cohomology ring of the dg endomorphism ring $\mathcal{SY}_{\Lambda/E}(M, M)$ is isomorphic to the \emph{singular Yoneda ring} of $M$:
$$\widehat{\rm Ext}^*_{\Lambda}(M, M)=\bigoplus_{i\in \mathbb{Z}} {\rm Hom}_{\mathbf{D}'_{\rm sg}(\Lambda)} (M, \Sigma^i(M)).$$
 Here, $\mathbf{D}'_{\rm sg}(\Lambda)$ is the big singularity category of $\Lambda$ defined in Remark~\ref{rem:bigsingularitycategory}.

(2) Assume that $\Lambda$ is an artin algebra with a decomposition $\Lambda=E\oplus J$, where $J$ is the Jacobson radical of $\Lambda$. Using Proposition~\ref{prop:sing-SY} and \cite[Theorem~9.5]{CW}, one relates $\mathbf{S}_{\rm dg}(\Lambda)$ to the dg perfect derived category of the corresponding \emph{dg Leavitt algebra}. We mention that Leavitt algebras appears already  in \cite{Leav}.
\end{rem}

Combining Propositions~\ref{prop:sing-V} and \ref{prop:sing-SY}, we obtain the following result.

\begin{cor}
There is an isomorphism $\mathcal{V}^{-, b}(\Lambda\mbox{-}{\rm proj})\simeq \mathcal{SY}^f_{\Lambda/E}$ in ${\bf Hodgcat}$.
\end{cor}

\begin{rem}
It seems  possible to construct a genuine dg functor $\mathcal{SY}^f_{\Lambda/E}\rightarrow\mathcal{V}^{-, b}(\Lambda\mbox{-}{\rm proj})$ inducing the isomorphism above. Roughly speaking, it sends a complex $X$ to its dg-projective resolution $\mathbb{B}\otimes_\Lambda X$. The details rely on a technical result \cite[Proposition~5.5]{CW2},  and  will appear elsewhere.
\end{rem}

\noindent{\bf Acknowledgements}.\quad We are very grateful to the referee for many helpful comments, and to Hongxing Chen, Bernhard Keller and Henning Krause for useful suggestions. X.W. thanks Luchezar Avramov, who encouraged him to write up the material concerning the Vogel dg category in a conference held in Utah, 2018. X.W. thanks the organizes of the twentieth ICRA for the opportunity to present these results.  The work is supported by  the National Natural Science Foundation of China (No.s 12325101, 12131015, and 12161141001) and the DFG grant (WA 5157/1-1).

\bibliography{}

\vskip 10pt

 {\footnotesize \noindent  Xiao-Wu Chen\\
 Key Laboratory of Wu Wen-Tsun Mathematics, Chinese Academy of Sciences,\\
 School of Mathematical Sciences, University of Science and Technology of China, Hefei 230026, Anhui, PR China\\

 \noindent Zhengfang Wang\\
 Department of Mathematics, Nanjing University, Nanjing 210093, Jiangsu, PR China\\
\vskip -5pt
 \noindent Institute of Algebra and Number Theory, University of Stuttgart\\
Pfaffenwaldring 57, 70569 Stuttgart, Germany \\
 }

\end{document}